\documentclass[11pt,reqno]{amsproc}
\usepackage[margin=1in]{geometry}
\usepackage{amsmath, amsthm, amssymb}
\usepackage{times, esint,stackrel}
\usepackage{enumitem, graphicx}
\usepackage{color}
\usepackage{enumitem}   
\usepackage{xcolor} 
\usepackage{etoolbox}
\usepackage{mathrsfs}
\usepackage{cite}
\usepackage{mathabx}
\usepackage{stackengine}
\usepackage{subcaption,graphicx}
\usepackage{float}
\usepackage{hyperref}
\usepackage{bbm}
\usepackage{tcolorbox}
\usepackage{tikz}
\usetikzlibrary{decorations.pathreplacing,decorations.markings}
\tikzset{
	on each segment/.style={
		decorate,
		decoration={
			show path construction,
			moveto code={},
			lineto code={
				\path [#1]
				(\tikzinputsegmentfirst) -- (\tikzinputsegmentlast);
			},
			curveto code={
				\path [#1] (\tikzinputsegmentfirst)
				.. controls
				(\tikzinputsegmentsupporta) and (\tikzinputsegmentsupportb)
				..
				(\tikzinputsegmentlast);
			},
			closepath code={
				\path [#1]
				(\tikzinputsegmentfirst) -- (\tikzinputsegmentlast);
			},
		},
	},
	mid arrow/.style={postaction={decorate,decoration={
				markings,
				mark=at position .5 with {\arrow[#1]{stealth}}
	}}},
}

\usepackage{mathtools}
\mathtoolsset{showonlyrefs}
\mathtoolsset{showonlyrefs=true}

\makeatletter

\def\@tocline#1#2#3#4#5#6#7{\relax
  \ifnum #1>\c@tocdepth % then omit
  \else
    \par \addpenalty\@secpenalty\addvspace{#2}%
    \begingroup \hyphenpenalty\@M
    \@ifempty{#4}{%
      \@tempdima\csname r@tocindent\number#1\endcsname\relax
    }{%
      \@tempdima#4\relax
    }%
    \parindent\z@ \leftskip#3\relax \advance\leftskip\@tempdima\relax
    \rightskip\@pnumwidth plus4em \parfillskip-\@pnumwidth
    #5\leavevmode\hskip-\@tempdima
      \ifcase #1
       \or\or \hskip 1em \or \hskip 2em \else \hskip 3em \fi%
      #6\nobreak\relax
    \dotfill\hbox to\@pnumwidth{\@tocpagenum{#7}}\par
    \nobreak
    \endgroup
  \fi}
\makeatother

\def\dd{{\rm d}}

\def\eps{\varepsilon}
\def\e{{\rm e}} 
\def\de{{\partial}}

\def\RR {\mathbb{R}}
\def\TT {\mathbb{T}}
\def\ZZ {\mathbb{Z}}

\def\Re{{\rm Re}}

\newcommand{\norm}[1]{\left\lVert #1 \right\rVert}

\newcommand{\jap}[1]{\left\langle #1 \right\rangle}

\newcommand{\sfE}{\mathsf{E}}
\newcommand{\sfD}{\mathsf{D}}
\newcommand{\sfG}{\mathsf{G}}
\newcommand{\sfT}{\mathsf{T}}
\newcommand{\sfS}{\mathsf{S}}
\newcommand{\sfR}{\mathsf{R}}

\newcommand{\sfsym}{\mathsf{sym}}
\newcommand{\sfho}{\mathsf{h.o.}}

\newcommand{\cS}{\mathcal{S}}
\newcommand{\cL}{\mathcal{L}}

\newcommand{\cF}{\mathcal{F}}

\newcommand{\cI}{\mathcal{I}}

\newcommand{\cT}{\mathcal{T}}

\newcommand{\cJ}{\mathcal{J}}

\newtheorem{proposition}{Proposition}[section]
\newtheorem{theorem}[proposition]{Theorem}

\newtheorem{lemma}[proposition]{Lemma}
\theoremstyle{definition}

\newtheorem{remark}[proposition]{Remark}

\theoremstyle{definition}

\numberwithin{equation}{section}
\title[Stability threshold of Couette flow in in 2D MHD]{Stability threshold of the 2D Couette flow in a homogeneous magnetic field using symmetric variables}
\author[M. Dolce]{Michele Dolce}
\address{Institute of Mathematics, EPFL, Station 8, 1015 Lausanne, Switzerland}
\email{michele.dolce@epfl.ch}

\begin{document}
\maketitle
\begin{abstract}
We consider a 2D incompressible and electrically conducting fluid in the domain $\mathbb{T}\times\mathbb{R}$. The aim is to quantify stability properties of the Couette flow $(y,0)$ with a constant homogenous magnetic field $(\beta,0)$ when $|\beta|>1/2$. The focus lies on the regime with small fluid viscosity $\nu$, magnetic resistivity $\mu$ and we assume that the magnetic Prandtl number satisfies $\mu^2\lesssim\mathrm{Pr}_{\mathrm{m}}=\nu/\mu\leq 1$. We establish that small perturbations around this steady state remain close to it, provided their size is of order $\eps\ll\nu^{2/3}$ in $H^N$ with $N$ large enough. Additionally, the vorticity and current density experience a  transient growth of order $\nu^{-1/3}$ while converging exponentially fast to an $x$-independent state after a time-scale of order $\nu^{-1/3}$. The growth is driven by an inviscid mechanism, while the subsequent exponential decay results from the interplay between transport and diffusion, leading to the \textit{dissipation enhancement}. A key argument to prove these results is to reformulate the system in terms of \textit{symmetric variables}, inspired by the study of inhomogeneous fluid, to effectively characterize the system's dynamic behavior.
\end{abstract}

\tableofcontents

\section{Introduction}
The 2D incompressible Navier-Stokes magnetohydrodynamics (NS--MHD) equations on the domain $\TT\times \RR$ are given by 
\begin{align}
\label{eq:utilde}
\begin{cases}\de_t\tilde{u}+\tilde{u}\cdot \nabla \tilde{u}-\tilde{b}\cdot \nabla \tilde{b}=\nu \Delta \tilde{u}-\nabla \tilde{p}, \qquad t>0,\, x\in \TT, \, y\in \RR,\\
\de_t\tilde{b}+\tilde{u}\cdot \nabla \tilde{b}-\tilde{b}\cdot \nabla \tilde{u}=\mu \Delta \tilde{b},\\
\nabla \cdot \tilde{u}=\nabla \cdot \tilde{b}=0\\
\tilde{u}|_{t=0}=\tilde{u}^{in},\qquad \tilde{b}|_{t=0}=\tilde{b}^{in}.
\end{cases}
\end{align}
Here $\tilde{u},\tilde{b}$ are the velocity and magnetic fields, $\tilde{p}$ the pressure, $\nu$ and $\mu$ are the fluid viscosity and magnetic resistivity, which are  proportional to the inverse Reynolds number and inverse magnetic Reynolds number respectively. We consider a nearly ideal system in the regime
\begin{equation}
0<\nu\leq \mu\ll 1, \qquad \Longrightarrow \qquad \mathrm{Pr}_{\mathrm{m}}=\nu/\mu\leq 1,
\end{equation}
where $\mathrm{Pr}_{\mathrm{m}}$ is the magnetic Prandtl number, observed to be of order $10^{-7}-10^{-2}$ in physically relevant cases \cite{riva2022methodology,schumacher2020colloquium}.

A steady state of \eqref{eq:utilde} is the Couette flow with a constant background magnetic field, that is
\begin{equation}
\label{couette}
u_E=(y,0), \qquad b_E=(\beta,0), \qquad \beta \in \mathbb{R}.
\end{equation}
This is problably one of the simplest setting to understand some quantitative  hydromagnetic stability properties of shear flows, which is a problem of significant physical interest \cite{chandrasekhar2013hydrodynamic,hughes2001instability,davidson2002introduction}.  The presence of a background magnetic field could dramatically change stability features of the shear flow considered: i) it can have a destabilizing effect for shear flows that are linearly stable without the magnetic field (as the Couette flow) \cite{chen1991sufficient,hughes2001instability,hirota2005resonance,zhai2021long,hussain2018instability}. ii) it can suppress instabilities as the Kelvin-Helmholtz one \cite{liu2018physical} or lift-up effects in 3D fluids \cite{liss2020sobolev}. 

In this paper, we focus  on quantifying a  \textit{stability threshold} in Sobolev spaces. Following \cite{BGM15III}, the problem can be formulated as follows: 
\medskip 

\begin{quote}\textbf{Stability threshold}: let $N\geq0$, $0<\nu\leq \mu<1$, $(\tilde{u}^{in},\tilde{b}^{in})=(u_E,b_E)+(u^{in},b^{in})$. What is the smallest $\gamma=\gamma(N)\geq 0$ such that if $\norm{(u^{in},b^{in})}_{H^N}= \eps<\nu^\gamma$ then $\norm{(u(t),b(t))}_{L^2}\ll 1$ and $(u(t),b(t))$ converges back to a laminar flow as $t\to\infty$?
\end{quote}

\medskip \noindent
Let us briefly review the literature about related problems. Since Reynolds's famous experiment \cite{reynolds1883xxix}, it is a classical problem in fluid dynamics to understand under which circumstances a laminar flow transitions to a turbulent state. Estimating a stability threshold is a quantitative way to establish when turbulence does not develop.  The idea that the laminar regime persists if the size of the perturbation decreases at large Reynolds number was already predicted by Kelvin in 1887 \cite{kelvin1887stability}. The quantification in terms of powers of the Reynolds number was also linked to the non-normal behavior of the linearized operator around a shear flow in the influential paper by Trefethen et al. \cite{trefethen1993hydrodynamic} and we refer to the book \cite{schmid2002stability} for further developments and references. In the last decade, there has been a significant effort in rigorously proving estimates for the Sobolev stability threshold in many different fluid problems involving the Couette flow \cite{BGM15III,liss2020sobolev,bedrossian2018sobolev,chen2020transition,masmoudi2022stability,wei2021transition,li2022asymptotic,masmoudi2020enhanced,zhai2023stability,li2022dynamical,zillinger2021enhanced} and recently strictly monotone shear flows as well \cite{li2023asymptotic}. Results in this direction are known also for the Poiseuille flow in $\TT\times \RR$ \cite{coti2020enhanced,del2021enhanced} or in $\TT\times[-1,1]$ with Navier-slip boundary conditions \cite{ding2022enhanced}, the Lamb-Oseen vortex \cite{gallay2018enhanced} and the Taylor-Couette flow \cite{an2021enhanced}. For Gevrey-regular perturbations, one can improve the stability threshold \cite{BGM15I,BGM15II,bedrossian2016enhanced} and even study problems in absence of viscosity \cite{masmoudi20bouss,BM15,BBCZD23,ionescu20,ionescu2019axi,Ionescu2020NonlinearID,masmoudi2020nonlinear}. In fact, the groundbreaking result by Bedrossian and Masmoudi \cite{BM15}, proving the \textit{nonlinear inviscid damping} around Couette in 2D Euler,  inspired many of the subsequent works involving strictly monotone shear flows.

For electrically conducting fluids, a stability threshold was first proved by Liss in \cite{liss2020sobolev} in 3D NS-MHD. For the 2D case, recently Zhao and Zi \cite{zhao2023asymptotic} studied the stability of \eqref{couette} with $\nu=0$ (2D Euler-MHD system)  with $|\beta|$ sufficiently large and perturbations of size $O(\mu)$ in the Gevrey-$1/2^{-}$ space. The latter regularity requirement might be necessary for the inviscid problem \cite{knobel2023echoes}. 
For what concerns the 2D NS-MHD system considered here, a Sobolev stability threshold $O(\nu^{5/6^+})$ was first proved by Chen and Zi \cite{chen23} for shear close to Couette when $\nu=\mu$, about which we comment more later on.

To state the main result for the problem studied in this paper, we first need to introduce the vorticity and the current density of the perturbation 
\begin{equation}
\omega=\nabla^\perp \cdot u, \qquad j=\nabla^\perp \cdot b.
\end{equation}
The system satisfied by $(\omega,j)$ is:
\begin{align}
\label{def:omjsystem}
\begin{cases}\de_t\omega+y\de_x\omega-\beta\de_xj-\nu \Delta \omega=\mathrm{NL}_\omega,\\
\de_tj+y\de_xj-\beta\de_x\omega-\mu \Delta j+2\de_{xy}\phi=\mathrm{NL}_j,\\
u=\nabla^\perp \psi, \qquad b=\nabla^\perp \phi,\\
\Delta\psi=\omega, \qquad \Delta \phi=j,\\
\omega|_{t=0}=\omega^{in},\qquad j|_{t=0}=j^{in},
\end{cases}
\end{align}
where 
\begin{align}
&\mathrm{NL}_\omega:=-u\cdot\nabla \omega+b\cdot \nabla j,\\
&\mathrm{NL}_j:=-u\cdot\nabla j+b\cdot \nabla \omega+2\de_{xy}\phi(\omega-2\de_{xx}\psi)-2\de_{xy}\psi(j-2\de_{xx}\phi).
\end{align}
In the following, we denote 
\begin{equation}
f_0(y)=\frac{1}{2\pi}\int_{\mathbb{T}}f(x,y)\dd x, \qquad f_{\neq}=f-f_0, \qquad \jap{a}:=\sqrt{1+a^2}
\end{equation}
We are ready to state the main result.
\begin{theorem}
\label{th:main}
Let $0<\nu\leq \mu\ll1$, $|\beta|>1/2$, $N>10$, and assume that $\nu\geq (16\mu/\beta^2)^3 $. Let $(\omega^{in},j^{in})$ be the initial data of \eqref{def:omjsystem}. Then, there exists a universal constant $0<\delta_0<1$ and $0<\eps_0=\eps_0(N,\beta)<\nu^\frac23$ such that for all $\eps<\eps_0$ the following holds true: if 
\begin{equation}
\norm{(\omega^{in},j^{in})}_{H^N}\leq \eps ,
\end{equation}
denoting $(\Omega,J)(t,x+yt,y)=(\omega,j)(t,x,y)$, we have 
\begin{align}
\label{bd:omj}&\norm{(\Omega_{\neq},J_{\neq})(t)}_{H^{N}}\lesssim\eps \jap{t}\e^{-\delta_0\nu^\frac13 t}\\
\label{bd:invdamp}&\norm{(u^1_\neq,b^1_{\neq})(t)}_{L^2}+\jap{t}\norm{(u^2_\neq,b^2_{\neq})(t)}_{L^2}\lesssim \eps \e^{-\delta_0\nu^\frac13 t},\\
\label{bd:u0b0}& \norm{(u_0,b_0)(t)}_{H^N}+\nu^\frac12\norm{\de_y(u_0,b_0)}_{L^2([0,t];H^N)}\lesssim \eps.
\end{align}
\end{theorem}
The bound in \eqref{bd:omj}, combines the linear in time transient growth with the dissipation enhancement of the vorticity and current density. The growth is  an inviscid linear mechanism generated by the background magnetic field, resulting in a transient amplification of order $\nu^{-1/3}$ in the viscous case. The exponential decay for times larger than $\nu^{-1/3}$ is a common feature of perturbations around the Couette flow \cite{liss2020sobolev,masmoudi2022stability,wei2021transition,li2022asymptotic,masmoudi2020enhanced,li2023asymptotic,zillinger2021enhanced}. This is caused by the interaction between the advection by $(y,0)$ with the diffusion: the Couette flow sends information towards high-vertical frequencies where dissipation is more efficient, leading to the accelerated decay of the non-zero horizontal frequencies. The estimates \eqref{bd:invdamp} are a direct consequence of \eqref{bd:omj}. They quantify the \textit{inviscid damping} \cite{BM15} of the second component of the velocity field, but we do not expect any decay on $u^1$ in view of the possible growth of the vorticity.  Finally, from \eqref{bd:u0b0} we deduce that the $x$-averages of the solution remain small so that the dynamics is effectively converging to a shear flow nearby the steady state \eqref{couette}.  Let us first sketch the strategy of proof and then we make a few remarks.
 \subsection*{Proof strategy:} a key point in the proof of Theorem \ref{th:main} is the use of the \textit{symmetric variables}: 
\begin{equation}
\label{def:zq}
z=(\de_{xx}(-\Delta))^{-1/2}\omega, \qquad q=(\de_{xx}(-\Delta))^{-1/2}j.
\end{equation}
These unknowns are inspired by an energy method introduced in \cite{antonelli2021linear} for compressible fluids, and further improved and refined in \cite{BCZD20,BBCZD23}. The main observation is that it is possible to ``symmetrize" the linearized system to get a new system that enjoys a better energy structure\footnote{The method allows to effectively capture some oscillations that are stabilizing the system.}, as  explained in detail in Section \ref{sec:linear}. For the linearized problem, we see that for $|\beta|>1/2$ we have a coercive energy functional for $(z,q)$, see \eqref{def:Esymlin}. Using the good properties of this energy, we prove that the results stated in Theorem \ref{th:main} are true at the linearized level, see Proposition \ref{prop:keylin}.

The idea is then to bootstrap the control of the linear energy functional to the nonlinear case, which is done in Sections \ref{sec:NL}-\ref{sec:NLproof}. There are  two main difficulties to overcome: 
\begin{enumerate}
\item  It is not straightforward to obtain bounds for the nonlinear system associated to $(z,q)$ because the inverse of $(\de_{xx}(-\Delta))^{-1/2}$ is not a uniformly bounded Fourier multiplier. This imply that we might encounter some derivatives losses when reconstructing the symmetric variables in the nonlinear terms. 
\item The symmetric variables do not provide enough information over the $x$-averages of the solution.
\end{enumerate}
To overcome these issues, we follow a strategy similar to the one used in \cite{BBCZD23}. For the first problem, we exploit the nice structure of the nonlinearity and $(\de_{xx}(-\Delta))^{-1/2}$. The main idea is that, by performing a frequency decomposition, we can exchange derivative losses with time-growth. This decomposition considers not only interactions between high-low (or low-high) frequencies but it also accounts for what occurs near Orr's critical time $t=\eta/k$, where $\eta$ and $k$ are the Fourier coefficients associated with the variables $x$ and $y$, respectively. The dissipation enhancenment plays a crucial role in avoiding the use of technically involved Fourier multipliers needed in inviscid problems, e.g. \cite{BM15,BBCZD23}. In fact, it will suffice to capture the dissipation enhancement along with a form of inviscid damping using standard (by now) Fourier multipliers that are uniformly bounded in $\nu$.

The second problem instead can be explained as follows. The transport nonlinearity will generate terms containing the $k=0$-modes, for instance 
\begin{equation}
u\cdot \nabla  \omega= u_{\neq}\cdot \nabla \omega_{\neq}+u_0^1\de_x \omega_{\neq}+u_{\neq}^2\de_y\omega_0,
\end{equation}
where $u_0=(u_0^1,0)$ thanks to the incompressibility condition. Heuristically, $u_0^1$ has roughly the same regularity as $z$ (one derivative less than vorticity), and therefore, we can hope to control it using information on $z$. On the other hand, $\partial_y\omega_0$ has higher regularity with respect to $z$, and even dissipation cannot assist us since $\partial_y\omega_0=-\partial_{yy}u_0^1$, which involves two derivatives more than $z$. This suggests the need to directly control the system for $(\omega,j)$, which has a worst energy structure. Nevertheless, through the control on $(z,q)$,  we show that a dangerous linear term can be easily controlled  leading to bounds in agreement with the linearized behavior of $(\omega_{\neq},j_{\neq})$. The energy associated to $(\omega,j)$ is at the highest level of regularity but is allowed to grow in time, meaning that the control of the nonlinearity is somewhat easier. 
It is important to note that in Theorem \ref{th:main} we have not stated the bounds for $\norm{(\omega_0,j_0)}_{L^\infty([0,t];H^N)}=\norm{(u^1_0,b^1_0)}_{L^\infty([0,t];H^{N+1})}$, which are indeed of order $\eps\jap{t}$, even though there is no growth mechanism for the $x$-averages in the linearized problem.

\begin{remark}
The standard auxiliary variables for the NS--MHD system are the Els\"asser variables \cite{bardos1988longtime,wei2017global}, corresponding to $e^{\pm}=\omega \pm j$. The system satisfied by $e^\pm$ also has a nice structure where one could exploit the integration-in-time trick used by Liss in the 3D problem \cite{liss2020sobolev}. This strategy is followed in the 2D case by Chen and Zi \cite{chen23} as well, where there are also the additional complicatons given by the more general form of the shear flow considered.  It appears that using $(z,q)$ has certain technical advantages, particularly in achieving the $\nu^{2/3}$ threshold and handling cases where $\nu\neq \mu$. We also mention that, in the result for the 2D Euler--MHD obtained by Zhao and Zi \cite{zhao2023asymptotic}, the main energy functional introduced by the authors  uses an approximated version of $(z,q)$. In particular, $(\de_{xx}(-\Delta))^{-1/2}$ is replaced by a Fourier multiplier whose inverse is bounded with a $\mu$-dependent constant (the weight $m$ in \cite{zhao2023asymptotic}).  The use of symmetric variables has proven to be a flexible approach \cite{antonelli2021linear,BCZD20,BBCZD23,zillinger2021enhanced}, which is in essence a carefully weighted Kawashima's type energy argument 
\cite{kawashima1984systems}.
\end{remark}
 \begin{remark}
 \label{rem:threshold}
 The $\nu^{2/3}$  threshold in Sobolev spaces\footnote{The use of $H^N$ with $N>10$ is certainly not optimal and it might be of interest to understand what are the critical Sobolev spaces, in a similar spirit of \cite{masmoudi2020enhanced,li2022dynamical,li2023asymptotic}}  can be heuristically justified as in \cite{zhai2023stability}. Namely, for the 2D NS case the best available threshold is $\nu^{1/3}$ \cite{masmoudi2022stability}. Here, the vorticity and current density are experiencing a growth of order $\nu^{-1/3}$ after which the dissipation enhancement kicks in. We would then require an extra $\nu^{1/3}$ smallness to keep everything in a perturbative regime even with this transient growth, which  is why we need to assume $\eps\ll \nu^{1/3+1/3}$.  
 
  On the other hand, the threshold $\nu^{5/6^+}$ obtained in \cite{chen23} can be related to the method of proof. Specifically, the control of nonlinear terms is inspired by \cite{bedrossian2018sobolev} where a $\nu^{1/2}$ threshold is obtained in the 2D NS setting. In this paper, we need to treat the nonlinear terms in a more refined way compared to \cite{bedrossian2018sobolev} to improve the threshold, relying on estimates that are closer to inviscid problems. We believe that the nice methods introduced in \cite{chen23} to handle shear flows close to Couette, could be combined with the strategy we use here to obtain the $\nu^{2/3^+}$ threshold for shear near Couette as well. 
  \end{remark}
 
 \begin{remark}
The case of different viscosity coefficients $\mu\neq \nu$ is generally more challenging to study compared  to the case $\mu=\nu$, see for instance  \cite{wei2017global}. For the energy method employed here, having $\mu\neq \nu$ does not pose any significant difficulty because we can exploit the dissipation enhancement to handle some linear errors arising from this anisotropy. This is precisely why we need to assume that $\mu^3\ll \nu \leq \mu$, and we anticipate that the problem becomes much more intricate in the opposite regime, as hinted by the limiting case $\nu=0$ investigated in \cite{zhao2023asymptotic}.  \end{remark}
 \begin{remark}
 All the constants hidden in the symbol $\lesssim$ degenerate as $|\beta|\to 1/2$, meaning that $\eps_0\to 0$ as $|\beta|\to 1/2$. This is related to the coercivity of the energy functional we use in the linearized problem, for which we need $|\beta|>1/2$. In fact, when $\beta=0$ the linearized system is not coupled anymore and the current density will have a growth of order $\jap{t}^2$ instead of $\jap{t}$, see also Remark \ref{rem:beta0}.
  \end{remark}

\subsection{Notation}
We introduce some notation used throughout the paper. For $a,b\in \mathbb{R}$, we define 
\begin{equation}
|a,b|:=|a|+|b|, \qquad \jap{a}=\sqrt{1+a^2}.
\end{equation}
We use the notation $a\lesssim b$ to indicate that there is a constant $C>0$, independent of the relevant paramenters $\nu,\mu$ such that $a\leq Cb$. Similarly, we say $a\approx b$ if $a\lesssim b$ and $b\lesssim a$.

We define the Fourier transform of a function $F$ as 
\begin{equation}
\hat{f}_{k}(\eta)=\cF(f)_k(\eta)=\frac{1}{2\pi}\iint_{\mathbb{T}\times \mathbb{R}} \e^{i(kx+\eta y)} f(x,y)\dd x \dd y,
\end{equation}
and the inverse Fourier transform as 
\begin{equation}
\cF^{-1}(\hat{f})(x,y)=\frac{1}{2\pi}\sum_{k\in \ZZ}\int_{\mathbb{R}} \e^{i(kx+\eta y)} \hat{f}_k(\eta)\dd \eta.
\end{equation}
We identify Fourier multipliers $w(\nabla)$ with their symbol $w_k(t,\eta)$, except for standard derivatives $\de_x, \de_y$ where we  use the symbols $ik, i\eta$. 
We denote the $L^2$ scalar product as 
\begin{equation}
\jap{f,g}_{L^2}=\langle{\hat{f},\hat{g}}\rangle_{L^2}=\sum_{k\in \mathbb{Z}}\int_{\mathbb{R}}\hat{f}_{k}(\eta)\bar{\hat{g}}_{k}(\eta)\dd \eta,
\end{equation}
and the norm in $H^N$ as 
\begin{equation}
\norm{f}_{H^N}^2=\sum_{k\in \mathbb{Z}}\int_{\mathbb{R}}\jap{|k,\eta|}^{2N}|\hat{f}_{k}(\eta)|^2\dd \eta.
\end{equation}
We use the following convention
\begin{equation}
\label{not:F}
\jap{\de_{xx}(f\de_{xx}g),h}_{L^2}=\jap{k^2(\hat{f}*(\ell^2\hat{g})),\hat{h}}_{L^2}=\sum_{k,\ell\in \mathbb{Z}}\iint_{\mathbb{R}^2}k^2\hat{f}_{k-\ell}(\eta-\xi)\ell^2\hat{g}_{\ell}(\xi)\bar{\hat{h}}_{k}(\eta)\dd \eta \dd \xi.
\end{equation}
We define the frequency decomposition as in \cite{BGM15III,liss2020sobolev}: let $\chi:\RR^4\to \RR$ be 
\begin{equation}
\chi(k,\eta,\ell,\xi)=\begin{cases} 1&\qquad \text{if }\, |k-\ell,\eta-\xi|\leq 2|\ell,\xi|\\
0&\qquad \text{otherwise}.
\end{cases}
\end{equation}
We use the paraproduct decomposition 
\begin{align}
\cF(fg)_k(\eta)=\,&\sum_{k,\ell \in \mathbb{Z}}\int_{\mathbb{R}}\hat{f}_{k-\ell}(\eta-\xi)\hat{g}_{\ell}(\xi)\chi(k,\eta,\ell,\xi)\dd \xi\\
&+\sum_{k,\ell \in \mathbb{Z}}\int_{\mathbb{R}}\hat{f}_{k-\ell}(\eta-\xi)\hat{g}_{\ell}(\xi)(1-\chi(k,\eta,\ell,\xi))\dd \xi\\
\label{def:LH}:=\,&\cF(f^{Lo}g^{Hi})+\cF(f^{Hi}g^{Lo}).
\end{align}
Notice that $|k,\eta|\leq 3|\ell,\xi|$ on the support of $\chi$ and $|k,\eta|\leq 3|k-\ell,\eta-\xi|/2$ on the support of $1-\chi$.

\section{Linearized problem}
\label{sec:linear}
In this section, we study in detail the simple linearized dynamics. First of all, we introduce the change of coordinates 
\begin{equation}
X=x-yt, \qquad Y=y.
\end{equation}
We denote the variables in the \textit{moving frame} with capital letters 
\begin{equation}
\label{def:OJ}
\Omega(t,X,Y)=\omega(t,x,y), \qquad J(t,X,Y)=j(t,x,y),
\end{equation}
and 
\begin{equation}
\label{def:DeltaL}
\nabla_L=(\de_X,\de_Y-t\de_X), \qquad \Delta_L=\de_{XX}+(\de_Y-t\de_X)^2.
\end{equation}
The linearized problem in the moving frame is
\begin{align}
&\de_t\Omega=\nu\Delta_L\Omega+\beta\de_XJ,\\
&\de_tJ=\mu\Delta_L J-2\de_X(\de_Y-t\de_X)\Delta^{-1}_LJ +\beta\de_X\Omega.
\end{align}
Taking the Fourier transform in both variables, defining the symbol associated to $-\Delta_L$ as 
\begin{equation}
\label{def:p}
p_k(t,\eta):=k^2+(\eta-kt)^2,
\end{equation}
we get
\begin{align}
\label{eq:linOm}
&\de_t\hat{\Omega}=-\nu p\hat{\Omega}+\beta ik\hat{J},\\
\label{eq:linJ}&\de_t\hat{J}=-\mu p\hat{J}+\frac{\de_t p}{p}\hat{J} +\beta ik\hat{\Omega},
\end{align}
where we omit the subscript $k$ to ease the notation. Multiplying the equation for $\hat{J}$ by $p^{-1}$, notice that
\begin{align}
\label{eq:hO}
&\de_t\hat{\Omega}=-\nu p\hat{\Omega}+\beta ik p(p^{-1}\hat{J}),\\
\label{eq:hJ}&\de_t(p^{-1}\hat{J})=-\mu p(p^{-1}\hat{J})+\frac{\beta ik}{p}\hat{\Omega}.
\end{align}
We briefly comment below on the inviscid case. Then we study the viscous problem with a flexible energy method that will be useful in the nonlinear analysis.

\medskip
\noindent $\bullet$ \textit{Case $\nu=\mu=0$:}
in absence of viscosity, we see that the $2\times 2$ non-autonomous system \eqref{eq:hO}-\eqref{eq:hJ} has almost an antisymmetric structure, but the time-dependence of the factor $p$  prevents the existence of an exact conserved quantity. To overcome this problem, we apply the \textit{symmetrization scheme} introduced in \cite{antonelli2021linear} and further developed in \cite{BCZD20,BBCZD23}. This amounts at finding two good unknowns for which we have an \textit{almost conserved} quantity. In this case, the symmetrization procedure suggests the use of the variables 
\begin{align}
\label{def:ZQ}
\begin{cases}\displaystyle Z_k(t,\eta)=\sqrt{\frac{k^2}{p_k(t,\eta)}}\hat{\Omega}_k(t,\eta), \qquad Q_k(t,\eta)=\sqrt{\frac{k^2}{p_k(t,\eta)}}\hat{J}_k(t,\eta), \qquad \text{for }k\neq 0,\\
Z_0(t,\eta)=Q_0(t,\eta)=0.
\end{cases}
\end{align}
In the original reference frame, these variables are exactly the $(z,q)=(\de_{xx}(-\Delta))^{-\frac12}(\omega,j)$ discussed in the Introduction \eqref{def:zq}. The system satisfied by $(Z,Q)$ is
\begin{align}
\frac{\dd}{\dd t}\begin{pmatrix}{Z}\\ {Q}\end{pmatrix}=\begin{pmatrix}-\frac12\frac{\de_t p}{p}& \beta ik \\
\beta ik & \frac12\frac{\de_t p}{p} \end{pmatrix}\begin{pmatrix}{Z}\\ {Q}\end{pmatrix}.
\end{align}
The energy functional is then given by
\begin{equation}
\tilde{E}_{\sfsym}(t):=\frac12\left(|{Z}|^2+|{Q}|^2-\frac{\de_tp}{\beta ikp}\Re({Z}\bar{{Q}})\right)(t).
\end{equation}
\begin{remark}
If $\de_tp$ and $p$ were constants, one would have that $\tilde{E}_{\sfsym}(t)$ is a conserved quantity, meaning that the dynamics lie in an ellipse in the $Z$--$Q$ plane. With time-dependent coefficients we only aim at showing that the dynamics remains in an annular region in the $Z$--$Q$ plane.
\end{remark} 
Having that
\begin{equation}
\frac{|\de_t p_k(t,\eta)|}{p_{k}(t,\eta)}=\frac{2|k(\eta-kt)|}{k^2+(\eta-kt^2)}\leq 1,
\end{equation}
we see that energy functional is coercive only when $|\beta|>1/2$. In particular 
\begin{equation}
\frac12\left(1-\frac{1}{2\beta}\right)\left(|Z|^2+|Q|^2\right)\leq \tilde{E}_{\sfsym}\leq \frac12\left(1+\frac{1}{2\beta}\right)\left(|Z|^2+|Q|^2\right)
\end{equation}
In fact, the coercivity of the energy functional is the only reason why we need to assume $|\beta|>1/2$. 

Taking the time derivative of $\tilde{E}_{\sfsym}$ and using a Gr\"onwall type estimate, it is not difficult to show that
\begin{equation}
\tilde{E}_{\sfsym}(0)\approx_{\beta-\frac12} \tilde{E}_{\sfsym}(t),
\end{equation}
meaning that all the constants degenerate when $|\beta| \to 1/2$.
\begin{remark}
\label{rem:beta0}
It might be natural that $|\beta|=1/2$ is a somewhat sharp threshold to observe the linear-in-time growth. For instance, when $\beta=0$ one can explicitly solve the system and obtain that $\Omega_\neq$ is conserved in time whereas $J_{\neq}\approx \jap{t}^2$. It seems reasonable that for $0<|\beta|<1/2$ one simply interpolates between the behavior at $\beta=0$ and $|\beta|>1/2$, in a similar fashion to what happens in the Boussinesq case at small Richardson's number \cite{YL18}.
\end{remark}
\medskip
\noindent $\bullet$ \textit{Case $0<\nu\leq \mu$:}
When viscosity is present, we aim at capturing the dissipation enhancement, that is the exponential decay on a time-scale of order $O(\nu^{-\frac13})$. This could be proved by using the energy functional $\tilde{E}_{\sfsym}$ and some algebraic manipulation. However, with the idea in mind of addressing the nonlinear problem, we prove the enhanced dissipation estimate with the help of some Fourier multipliers, which are by now standard in the literature.
We use the following weights: the first one, introduced in \cite{Zillinger17}, is to control error terms which are integrable in time and is given by 
\begin{align}
\label{def:md}
\begin{cases}\displaystyle\de_tm^{d}_k=\frac{C_\beta}{1+(\eta/k-t)^2}m^{d}_k, \qquad \text{for } k\neq 0,\\
m^d_k(0,\eta)=1\\
m^d_0(t,\eta)=1,
\end{cases}
\end{align}
where $C_\beta>0$ is a fixed constant that can be chosen to be $C_{\beta}=\max\{1,4/|\beta|\}$ for example.
This weight is needed to recover some time-decay from the inviscid damping, that is generated by inverse powers of the Laplacian in the moving frame. Notice that 
\begin{equation}
\label{bd:md}
m^d_k(t,\eta)\approx 1 \qquad \text{ for all } t>0, \, \eta \in \mathbb{R}, \, k\neq 0.
\end{equation}
The next weight, introduced in \cite{BGM15I}, is needed to capture the dissipation enhancement and is defined as 
\begin{align}
\label{def:mnu}
\begin{cases}\displaystyle\de_tm^{\nu}_k=\frac{\nu^{\frac13}}{1+\nu^\frac23(\eta/k-t)^2}m^{\nu}_k, \qquad \text{for } k\neq 0, \, \nu>0,\\
m^\nu_k(0,\eta)=1\\
m^0_k(t,\eta)=m^\nu_0(t,\eta)=1.
\end{cases}
\end{align}
Also in this case we have 
\begin{equation}
\label{bd:mnu}
m^\nu_k(t,\eta)\approx 1 \qquad \text{ for all } t>0, \, \eta \in \mathbb{R}, \, k\neq 0.
\end{equation}
The key property of the weight $m^\nu$  is 
\begin{equation}
\label{eq:keymnu}
\nu p_k(t,\eta)+\frac{\de_t m^{\nu}_k(t,\eta)}{m^{\nu}_k(t,\eta)}\geq \frac14 \nu^\frac13, \qquad \text{ for all } t>0, \, \eta \in \mathbb{R}, \, k\neq 0,
\end{equation}
which can be easily checked by considering $|\eta/k-t|\leq \nu^{-\frac13}$ or $|\eta/k-t|\geq \nu^{-\frac13}$ separately.
This weight is compensating the inefficiency of the dissipation enhancement close to the critical time $t= \eta/k$. 

Finally, we need a last weight to absorb some error terms given by the mixed scalar product in the energy functional,  
\begin{align}
\label{def:ms}
\begin{cases}\displaystyle\de_tm^{s}_k=\frac{\gamma_\beta C_{\beta}}{(1+(\eta/k-t)^2)^{\frac32}}m^{s}_k, \qquad \text{for } k\neq 0, \, \nu>0,\\
m^\nu_k(0,\eta)=1\\
m^0_k(t,\eta)=m^\nu_0(t,\eta)=1,
\end{cases}
\end{align}
where $\gamma_\beta$ is a fixed constant such that 
\begin{equation}
\frac{1}{|\beta|}\left(\frac12+\frac{1}{\gamma_\beta}\right)<1.
\end{equation}
Notice that $\gamma_\beta \to +\infty$ as $|\beta|\to 1/2$. This weight is again bounded above and below, namely  
\begin{equation}
\label{bd:ms}
m^s_k(t,\eta)\approx 1 \qquad \text{ for all } t>0, \, \eta \in \mathbb{R}, \, k\neq 0,
\end{equation}
and satisfies 
\begin{equation}
\label{bd:ms}
\frac{\de_tm^s}{m^s}=\gamma_\beta\sqrt{\frac{k^2}{p}}\frac{\de_t m^d}{m^d}.
\end{equation}

Aiming at obtaining a bound in Sobolev spaces, we then define the weight 
\begin{equation}
\label{def:m}
m_k(t,\eta)=\begin{cases}\displaystyle\e^{\delta_0\nu^\frac13 t}\jap{|k,\eta|}^N(m^dm^\nu m^s)^{-1}_k(t,\eta), &\qquad \text{for } k\neq 0,\\
\jap{\eta}^N &\qquad \text{for }k= 0,
\end{cases}
\end{equation}
where $0<\delta_0<1/64$ is a sufficiently small constant chosen later in the proof.
When $k\neq 0$ we have
\begin{equation}
\label{eq:dtmlin}
\frac{\de_t m_k}{m_k}=\delta_0\nu^\frac13-\sum_{\iota \in \{\nu,d,s\}}\frac{\de_t m^\iota_k}{m^\iota_k}
\end{equation}
 The good unkowns are still given by \eqref{def:ZQ}. The system for the weighted variables $(mZ,mQ)$ read as
\begin{align}
&\de_t(mZ)=-\left(\nu p-\frac{\de_t m}{m}\right)mZ-\frac12\frac{\de_t p}{p}mZ+\beta ik mQ,\\
&\de_t(mQ)=- \left(\mu p-\frac{\de_t m}{m}\right)mQ+\frac12\frac{\de_t p}{p}mQ+\beta i kmZ.
\end{align}
The energy functional associated to the system is
\begin{equation}
\label{def:Esymlin}
E_{\sfsym}(t):=\frac12\left(|mZ|^2+|mQ|^2-\Re\left(\frac{\de_tp}{\beta ikp}(mZm\bar{Q})\right)\right)(t).
\end{equation}
We have the following.
\begin{proposition}
\label{prop:keylin}
Let $0<\nu\leq \mu\ll 1$, $|\beta|>1/2$ and assume that assume that $\nu^\frac13\geq (16\mu/\beta^2)$. Then
 \begin{equation}
 \label{bd:Elinsym}
E_{\sfsym}(t)+\frac{1}{16}\int_0^tD_{\sfsym}(\tau)\dd \tau\leq E_{\sfsym}(0),
\end{equation}
where 
\begin{align}
\label{def:Dsymlin}
D_{\sfsym}(t):= \left(\nu p|mZ|^2+\mu p|mQ|^2+\left(\frac{\de_t m^\nu}{m^\nu}+\frac{\de_t m^d}{m^d}\right)(|mZ|^2+|mQ|^2)\right)(t).
\end{align}

As a consequence of this bound, the following inequalities holds true:
\begin{align}
\label{bd:ZQlin}&\norm{Z(t)}_{H^N}+\norm{Q(t)}_{H^N}\lesssim_{\beta-\frac12} \e^{-\delta_0\nu^\frac13 t}\left(\norm{Z^{in}}_{H^N}+\norm{Q^{in}}_{H^N}\right),\\
\label{bd:OJlin}&\norm{\Omega_{\neq}(t)}_{H^N}+\norm{J_{\neq}(t)}_{H^N}\lesssim_{\beta-\frac12} \jap{t}\e^{-\delta_0\nu^\frac13 t}\left(\norm{\omega_{\neq}^{in}}_{H^{N}}+\norm{j_{\neq}^{in}}_{H^{N}}\right),\\
\label{bd:inda}&\norm{U^1_{\neq}(t)}_{H^N}+\jap{t}\norm{U^2_{\neq}(t)}_{H^{N-1}}\lesssim_{\beta-\frac12} \e^{-\delta_0\nu^\frac13 t}\left(\norm{\omega_{\neq}^{in}}_{H^{N}}+\norm{j_{\neq}^{in}}_{H^{N}}\right).
\end{align}
%\begin{remark}
%To get the inviscid damping estimates \eqref{bd:inda}, we exploit the bounds on the symmetric variables $(Z,Q)$ instead of using \eqref{bd:OJlin}, which would require to pay more regularity on the initial data.
%\end{remark}
\end{proposition}
\begin{proof}
Compute that
\begin{align}
\label{eq:dmZ}&\frac12\frac{\dd}{\dd t}|mZ|^2=-\left(\nu p-\frac{\de_t m}{m}\right)|mZ|^2+\delta_0\nu^\frac13|mZ|^2-\frac12\frac{\de_t p}{p}|mZ|^2+\beta \Re(ik mQ m\bar{Z}),\\
\label{eq:dmQ}&\frac12\frac{\dd}{\dd t}|mQ|^2=-\left(\mu p-\frac{\de_t m}{m}\right)|mQ|^2+\delta_0\nu^\frac13|mQ|^2+\frac12\frac{\de_t p}{p}|mQ|^2+\beta \Re(ik mZ m\bar{Q}).
\end{align}
When adding these two equations the last terms on the right-hand side cancel out. For the mixed product instead, we have
\begin{align}
\label{eq:mixed}-\frac12\frac{\dd }{\dd t}\left(\frac{\de_tp}{\beta ikp}(mZm\bar{Q})\right)=\,&-\frac12\frac{\de_t p}{p}\left(|mQ|^2-|mZ|^2\right)\\
&\,+\frac{\de_t p}{2\beta ikp}\left((\nu+\mu)p-2\delta_0\nu^\frac13+2\sum_{\iota\in \{\nu,d,s\}}\frac{\de_t m^\iota}{m^\iota}\right) (mZm\bar{Q})\\
&-\frac12\left(\frac{p\de_{tt}p-(\de_t p)^2}{\beta ikp^2}\right)(mZm\bar{Q}).
\end{align}
Observe that the first term on the right-hand side of \eqref{eq:mixed} cancel out with the second to last terms in \eqref{eq:dmZ}-\eqref{eq:dmQ} when computing the time derivative of $E_{\sfsym}$. Thanks to the energy identities above, the property \eqref{eq:dtmlin} and the definition \eqref{def:Dsymlin} we arrive at the following inequality
\begin{equation}
\label{bd:Elin1}
\frac{\dd }{\dd t}E_{\sfsym}+D_{\sfsym}+\frac{\de_t m^s}{m^s}(|mZ|^2+|mQ|^2)\leq \sum_{i=0}^{5}\cL_i,
\end{equation}
where we define the linear error terms as: 
\begin{align}
&\cL_0:=\delta_0 \nu^\frac13\left(1+\frac{|\de_t p|}{|\beta||k|p} \right) (|mZ|^2+|mQ|^2),\\
&\cL_1:=(\nu+\mu)\frac{ |\de_t p|}{2|\beta||k|} |mZ||mQ|, \\ 
&\cL_2:=\frac{|\de_t p|}{|\beta||k|p}\frac{\de_tm^{\nu}}{m^{\nu}}|mZ||mQ|,\\
&\cL_3:=\left(\frac{p|\de_{tt}p|+(\de_t p)^2}{2|\beta| |k|p^2}\right)|mZ||mQ|,\\
&\cL_4:=\frac{|\de_t p|}{|\beta||k|p}\frac{\de_t m^d}{m^d}|mZ||mQ|,\\
&\cL_5:=\frac{|\de_t p|}{|\beta||k|p}\frac{\de_t m^s}{m^s}|mZ||mQ|.
\end{align}
Using \eqref{eq:keymnu}, we get 
\begin{equation}
\label{bd:L0}
\cL_0\leq 8 \delta_0D_{\sfsym},
\end{equation}
where we also used that $\mu\geq \nu$. For $\cL_1$, since $$\frac{|\de_t p|}{|k|}\leq 2\sqrt{p},$$ we have 
\begin{equation}
\cL_1\leq \left(\frac{\nu}{4}p +\frac{\mu}{\beta^2}\right)|mZ|^2+\left(\frac{\mu}{4}p +\frac{\nu}{\beta^2}\right)|mQ|^2.
\end{equation}
Now, we combine the hypothesis $\mu/\beta^2\leq \nu^\frac13/16$ and the property \eqref{eq:keymnu}
to get 
\begin{equation}
\label{bd:L1}
\cL_1\leq   \frac12 D_{\sfsym}.
\end{equation}
For $\cL_2$, combining  
\begin{equation}
\label{bd:triv1}
\frac{|\de_t p|}{|\beta||k|p}\leq \frac{2}{|\beta|\sqrt{p}}\leq \frac{2}{|\beta|C_\beta}\sqrt{\frac{\de_tm^d}{m^d}},
\end{equation}
with $\de_tm^\nu/m^\nu\leq \nu^\frac13$, we obtain
\begin{equation}
\cL_2\leq \frac{1}{64}\frac{\de_t m^\nu}{m^\nu}|mZ|^2+ \frac{64\nu^\frac13}{|\beta|^2C_\beta^2}\frac{\de_tm^d}{m^d}|mQ|^2.
\end{equation}
Since $\nu\ll1$, we have 
\begin{equation}
\label{bd:L2}
\cL_2\leq   \frac{1}{64} D_{\sfsym}.
\end{equation}
Turning our attention to $\cL_3$, observe that 
\begin{equation}
\frac{p|\de_{tt}p|+(\de_t p)^2}{2|\beta| |k|p^2}\leq \frac{|k|}{|\beta|p}\leq \frac{1}{|\beta|C_\beta}\frac{\de_tm^d}{m^d}.
\end{equation}
Hence
\begin{equation}
\label{bd:L3}
\cL_3\leq   \frac{1}{|\beta|C_\beta} D_{\sfsym}.
\end{equation}
To control $\cL_4$, combining the first bound in \eqref{bd:triv1} with the property \eqref{bd:ms} we have
\begin{align}
\label{bd:L4}
\cL_4\leq \frac{2}{|\beta|\sqrt{p}}\frac{\de_tm^d}{m^d}|mZ||mQ|\leq \frac{2}{|\beta|\gamma_\beta}\frac{\de_tm^s}{m^s}|mZ||mQ|\leq \frac{1}{|\beta|\gamma_\beta}\frac{\de_tm^s}{m^s}(|mZ|^2+|mQ|^2).
\end{align}
On the other hand, for $\cL_5$ we use $|\de_t p|/(|k\beta|p)\leq 1/|\beta|$ to get 
\begin{align}
\label{bd:L5}
\cL_5\leq  \frac{1}{2|\beta|}\frac{\de_tm^s}{m^s}(|mZ|^2+|mQ|^2).
\end{align}
Choosing $\delta_0,C_\beta,\gamma_\beta$ such that 
\begin{equation}
\delta_0<\frac{1}{64}, \qquad \frac{1}{|\beta|C_\beta} \leq \frac14, \qquad \frac{1}{|\beta|}\left(\frac12+\frac{1}{\gamma_\beta}\right)<1,
\end{equation}
which is always possible since $|\beta|>1/2$, we can combine \eqref{bd:L0}, \eqref{bd:L1}, \eqref{bd:L2}, \eqref{bd:L3}, \eqref{bd:L4} and \eqref{bd:L5} with \eqref{bd:Elin1} to get
\begin{equation}
\label{bd:Elin1}
\frac{\dd }{\dd t}E_{\sfsym}+\frac{1}{16}D_{\sfsym}+\left(1-\frac{1}{2|\beta|}-\frac{1}{|\beta|\gamma_\beta}\right)\frac{\de_t m^s}{m^s}(|mZ|^2+|mQ|^2)\leq 0,
\end{equation}
whence proving \eqref{bd:Elinsym}.

The bound \eqref{bd:ZQlin} is a straighforward consequence of the coercivity of $E_{\sfsym}$ and \eqref{bd:Elinsym}. To prove \eqref{bd:OJlin}, we can perform an energy estimate directly on \eqref{eq:linOm}--\eqref{eq:linJ} to get 
\begin{equation}
\label{bd:dtOJlin}
\frac12\frac{\dd }{\dd t}(|m\hat{\Omega}_{\neq}|^2+|m\hat{J}_{\neq}|^2)\leq  \frac{|\de_t p|}{p}|m\hat{J}_{\neq}|^2=\frac{|\de_t p|}{|k|\sqrt{p}}|mQ||m\hat{J}_{\neq}|\leq2 |mQ||m\hat{J}_{\neq}|.
\end{equation}
This inequality implies that 
\begin{align}
(|m\hat{\Omega}_{\neq}|+|m\hat{J}_{\neq}|)(t)&\lesssim (|m\hat{\Omega}_{\neq}|+|m\hat{J}_{\neq}|)(0)+\int_0^t|mQ|(\tau)\dd \tau\\
&\lesssim (|m\hat{\Omega}_{\neq}|+|m\hat{J}_{\neq}|)(0)+t \sqrt{E_\sfsym(0)}\\
&\lesssim \jap{t}(|m\hat{\Omega}_{\neq}|+|m\hat{J}_{\neq}|)(0),
\end{align}
where we used the coercivity of $E_{\sfsym}$ and the fact that $|\sqrt{k^2/p}\hat{F}|\leq |\hat{F}|$. Integrating in space and exploiting the definition of $m$ \eqref{def:m}, we deduce \eqref{bd:OJlin}.

The estimates \eqref{bd:inda}, follows by
\begin{align}
\label{bd:U1triv}
&|\hat{U}^1_{\neq}|=\frac{|\eta-kt|}{k^2+(\eta-kt)^2}|\hat\Omega_{\neq}|=\frac{|\eta/k-t|}{|k|\sqrt{p}}|Z|\leq |Z|\\
\label{bd:U2triv}&|\hat{U}^2_{\neq}|=\frac{|k|}{k^2+(\eta-kt)^2}|\hat\Omega_{\neq}|=\frac{1}{|k|\sqrt{p}}|Z|\leq \frac{\jap{|k,\eta|}}{\jap{t}}|Z|,
\end{align}
where in the last bound we used the general bound $\jap{a-b}\jap{b}\gtrsim \jap{a}$. Integrating in space we obtain the desired bound and conclude the proof.
\end{proof}

\section{Nonlinear problem}
\label{sec:NL}
For the nonlinear problem, the idea is to propagate the linearized behavior for the symmetric variables $(Z,Q)$, see \eqref{def:ZQ}, proved in Proposition \ref{prop:keylin}. As explained in the introduction, to overcome problems related to the $x$-averages (especially for $\de_y(\omega_0,j_0)$), 
%these variables give an information only about the non-zero $x$-frequencies and therefore we need to control the $x$-averages separately. In particular, for a typical transport transport term of the nonlinearity we observe the following: for any divergence free vector field $v=(v^1,v^2)$ we have $v_0^2=0$ and 
%\begin{equation}
%(v\cdot \nabla f)_{\neq}= (v_{\neq}\cdot \nabla f_{\neq})_{\neq}+v_0^1\de_x f_{\neq}+v_{\neq}^2\de_yf_0.
%\end{equation}
%In our case, $v\in \{u,b\}$  and $f\in \{\omega,j\}$. Therefore, we need to control $(u_0^1,b_0^1)$ and $(\de_y\omega_0,\de_yj_0)$. Estimates on $(u_0^1,b_0^1)$ can be obtained from the bounds on $(Z,Q)$ because they are, essentially, at the same level of regularity (recall that $(Z,Q)$ are basically one derivative less that $(\Omega,J)$, even though in an anisotropic way). On the other hand, $(\de_y\omega_0,\de_yj_0)=(\de_{yy}u^1_0,\de_{yy}b^1_0)$ are at a higher level of regularity and we cannot hope to control them only with the available information on $(Z,Q)$, even exploiting the dissipation. 
 we need to directly control  also $(\Omega,J)$.  As shown in the proof of Proposition \ref{prop:keylin}, we can use the bounds on $(Z,Q)$ to handle the problematic linear error term associated to $2\de_{xy}\Delta^{-1}j$ in the equation for $j$. Indeed, 
from the linearized problem \eqref{eq:linOm}-\eqref{eq:linJ}, as in  \eqref{bd:dtOJlin} we notice that  
\begin{equation}
\label{bd:linOmJ}
\frac12\frac{\dd}{\dd t}\left(\norm{\Omega}_{L^2}^2+\norm{J}^2_{L^2}\right)\leq \left|\jap{\frac{\de_t p}{p}\hat{J},\hat{J}}\right|=\left|\jap{\frac{\de_t p}{|k|\sqrt{p}}Q,\hat{J}}\right|\leq 2\norm{Q}_{L^2}\norm{J}_{L^2},
\end{equation}
where we used $|\de_t p/(|k|\sqrt{p})|\leq 2$. If we are able to propagate smallness on $Q$, say $\norm{Q}\lesssim \eps$ for some norm, and $\norm{J}\lesssim \eps \jap{t}$,  we can treat this term as forcing term of order $\eps^2 \jap{t}$ which would lead to bounds on $(\Omega,J)$ of order $\eps \jap{t}$ when integrating in time. This behavior is consistent with the growth observed in the linearized problem.

Before introducing the main ingredients for the proof of Theorem \ref{th:main}, we first rewrite the system \eqref{def:omjsystem} in the moving frame $X=x-yt,\, Y=y$. Recalling the notation \eqref{def:OJ}-\eqref{def:DeltaL}, we get 
\begin{align}
\label{def:OJsystem}
\begin{cases}\de_t\Omega-\beta\de_XJ-\nu \Delta_L \Omega=\mathrm{NL}_\Omega,\\
\de_tJ-\beta\de_X\Omega-\nu \Delta_L J+2\de_{X}(\de_Y-t\de_X)\Phi=\mathrm{NL}_J,\\
U=\nabla^\perp_L \Psi, \qquad B=\nabla^\perp_L J\\
\Delta_L\Psi=\Omega, \qquad \Delta_L \Phi=J,
\end{cases}
\end{align}
where 
\begin{align}
\label{def:NLO}&\mathrm{NL}_\Omega=-\nabla^\perp\Psi\cdot\nabla \Omega+\nabla^\perp\Phi\cdot \nabla J,\\
\label{def:NLJ}&\mathrm{NL}_J=-\nabla^\perp\Psi\cdot\nabla J+\nabla^\perp\Phi\cdot \nabla \Omega\\
\label{def:NLJS}&\qquad \quad+(2\de_{X}(\de_Y-t\de_X)\Phi)(\Omega-2\de_{XX}\Psi)-(2\de_{X}(\de_Y-t\de_X)\Psi)(J-2\de_{XX}\Phi).
\end{align}
\begin{remark}
Observe that we used the following crucial cancellation 
\begin{equation}
\nabla^\perp_L F\cdot \nabla_L G=\nabla F\cdot \nabla G,
\end{equation}
that is true for any  function $F,G$.
\end{remark}
From Proposition \ref{prop:keylin}, it is clear that the proof of Theorem \ref{th:main} is reduced in obtaining bounds for energy functionals controlling $(Z,Q)$ and $(\Omega,J)$.
\subsection{Energy functionals and the bootstrap scheme}
To introduce the energy functionals needed to prove the main Theorem \ref{th:main}, we  recall the definitions of  the symmetric variables $(Z,Q)$ and the weight $m$   respectively given in \eqref{def:ZQ} and \eqref{def:m}. 

The first energy functional is the one used in the linearized problem to control the symmetric variables. Here we cannot do estimates at fixed frequencies $(k,\eta)$ and therefore we define 
\begin{align}
\label{def:Esym}
\sfE_{\sfsym}(t)=\frac12 \left(\norm{mZ}^2_{L^2}+\norm{mQ}^2_{L^2}-\frac{1}{\beta}\Re\jap{\frac{1}{ik}\frac{\de_t p}{p}mZ,mQ}_{L^2}\right).
\end{align}
The goal is to propagate the smallness of this energy, namely $\sfE_{\sfsym}\lesssim \eps^2$ where $\eps$ is the size of the initial data in Theorem \ref{th:main}.

Then, we need the higher order energy to control directly the vorticity and current density, that is
\begin{align}
\label{def:Eho}
\sfE_{\sfho}(t)=\frac12 \left(\norm{m\Omega}^2_{L^2}+\norm{mJ}^2_{L^2}\right)
\end{align}
From the estimate \eqref{bd:linOmJ} and \eqref{bd:OJlin}, we expect that $\sfE_{\sfho}\lesssim \jap{t}^2\eps^2$.

Finally, to control the $x$-averages (which is also the only reason why we introduce $\sfE_{\sfho}$), we define
\begin{align}
\label{def:E0}
\sfE_{0}(t):=\frac12\left(\norm{U_0^1}^2_{H^N}+\norm{B_0^1}_{H^N}^2+\frac{1}{\jap{t}^2}\left(\norm{\Omega_0}^2_{H^N}+\norm{J_0}^2_{H^N}\right)\right)
\end{align}
We aim at propagating smallness for this functional, that is $\sfE_0\lesssim \eps^2$.
\begin{remark}
Notice that we allow the higher order zero modes $(\Omega_0,J_0)$, controlled by $\sfE_{\sfho}$ and $\sfE_0$, to grow linearly in time in $H^N$.
 One might expect to achieve uniform boundedness for $(\Omega_0,J_0)$. However, since they are at the highest level of regularity, controlling them requires using bounds on $(Z,Q)$ which are at  lower regularity. Essentially, we are trading regularity for time-growth, which is a standard argument in these types of energy estimates.
The inclusion of the term $t^{-1}(\Omega_0, J_0)$ in $\sfE_0$ is for technical purposes and does not provide any substantial information beyond what we already know from $\sfE_{\sfho}$.
\end{remark}
Before computing the time-derivative of the energy functionals, we introduce some notation. We define the \textit{good terms} as
\begin{align}
\label{def:Gnud}
\sfG_{\nu}[F]:=\norm{\sqrt{\frac{\de_t m^\nu}{m^\nu}} m F}^2_{L^2}, \qquad \sfG_{d}[F]:=\norm{\sqrt{\frac{\de_ tm^d}{m^d}} mF}^2_{L^2},
\end{align} 
which naturally arise from the time-derivative of the weight $m$. Associated to each energy functional, we have the dissipation functionals defined as
\begin{align}
\label{def:Dsym}
&\sfD_{\sfsym}(t):=\nu\norm{\nabla_L mZ}^2_{L^2}+\mu\norm{\nabla_Lm Q}^2_{L^2}+\sum_{\iota\in\{\nu,d\}}\sfG_\iota[Z]+\sfG_\iota[Q],\\
\label{def:Dho}&\sfD_{\sfho}(t):=\nu\norm{\nabla_L m\Omega}^2_{L^2}+\mu\norm{\nabla_L mJ}^2_{L^2}+\sum_{\iota\in\{\nu,d\}}\sfG_\iota[\Omega]+\sfG_\iota[J],\\
\label{def:D0}&\sfD_0(t):=\left(\nu\norm{\de_YU_0^1}^2_{H^N}+\mu\norm{\de_YB_0^1}^2_{H^N}+\frac{1}{\jap{t}^2}\left(\nu\norm{\de_Y\Omega_0}^2_{H^N}+\mu\norm{\de_YJ_0}^2_{H^N}\right) \right).
\end{align}
We are now ready to compute some basic energy inequalities where we exploit the bounds obtained in the linearized problem and we introduce the nonlinear error terms.
\begin{lemma}
\label{lem:enid}
Let $0<\nu\leq \mu\ll 1$, $|\beta|>1/2$ and assume that assume that $\nu^\frac13\geq (16\mu/\beta^2)$. 
Let $\sfE_{\iota},\sfD_{\iota}$ with $\iota\in\{\sfsym,\sfho,0\}$ be the energy and dissipation functionals defined in \eqref{def:Esym}-\eqref{def:E0} and \eqref{def:Dsym}-\eqref{def:D0}. Then, 
\begin{align}
\label{bd:dtEsym}
&\frac{\dd }{\dd t}\sfE_{\sfsym}+\frac{1}{16}\sfD_{\sfsym}\leq \sfT_{\sfsym}+\sfS_{\sfsym},\\
\label{bd:dtEho}
&\frac{\dd }{\dd t}\sfE_{\sfho}+\frac{1}{16}\sfD_{\sfho}\leq 4\sqrt{\sfE_{\sfsym}}\sqrt{\sfE_{\sfho}}+\sfT_{\sfho}+\sfS_{\sfho},\\
\label{bd:dtE0}
&\frac{\dd }{\dd t}\sfE_{0}+\sfD_{0}\leq \sfR_{\neq},
\end{align}
where, with the convention introduced in \eqref{not:F}, we define the following error terms: 
%denoting 
%\begin{equation}
%\label{def:w}
%w_k(t,\eta):=\sqrt{\frac{k^2}{p_k(t,\eta)}}m_k(t,\eta) \qquad \Longrightarrow \qquad (w\hat{\Omega},w\hat{J})=(mZ,mQ),
%\end{equation}
the error for the symmetric variables are given by
%\begin{align}
%\label{def:Tsym}
%\sfT_{\sfsym}:=\,&\cT_{\sfsym}\left(\Psi,\Omega,mZ+\frac{1}{i\beta k}\frac{\de_tp}{p}mQ\right)+\cT_{\sfsym}\left(\Phi,J,mZ-\frac{1}{i\beta k}\frac{\de_tp}{p}mQ\right)\\
%&+\cT_{\sfsym}\left(\Psi,J,mQ+\frac{1}{i\beta k}\frac{\de_tp}{p}mZ\right)+\cT_{\sfsym}\left(\Phi,\Omega,mQ-\frac{1}{i\beta k}\frac{\de_tp}{p}mZ\right)
%\end{align}
\begin{align}
\label{def:Tsym}
\sfT_{\sfsym}:=&\left|\jap{\sqrt{\frac{k^2}{p}}m\cF\left(\nabla^\perp \Psi\cdot \nabla \Omega\right),mZ+\frac{1}{i\beta k}\frac{\de_tp}{p}mQ}\right|\\
&+\left|\jap{\sqrt{\frac{k^2}{p}}m\cF\left(\nabla^\perp \Phi\cdot \nabla J\right),mZ-\frac{1}{i\beta k}\frac{\de_tp}{p}mQ}\right|\\
&+\left|\jap{\sqrt{\frac{k^2}{p}}m\cF\left(\nabla^\perp \Psi\cdot \nabla J\right),mQ+\frac{1}{i\beta k}\frac{\de_tp}{p}mZ}\right|\\
&+\left|\jap{\sqrt{\frac{k^2}{p}}m\cF\left(\nabla^\perp \Phi\cdot \nabla \Omega\right),mQ-\frac{1}{i\beta k}\frac{\de_tp}{p}mZ}\right|,
\end{align}
and
\begin{align}
\label{def:Ssym}
\sfS_{\sfsym}:=\,&\left|\jap{\sqrt{\frac{k^2}{p}}m\bigg(\frac{\de_t p}{p}\hat{J}*\big(\hat{\Omega}-2\frac{\ell^2}{p}\hat{\Omega}\big)\bigg),mQ-\frac{1}{i\beta k}\frac{\de_t p}{p}mZ}\right|\\
&+\left|\jap{\sqrt{\frac{k^2}{p}}m\bigg(\frac{\de_t p}{p}\hat{\Omega}*\big(\hat{J}-2\frac{\ell^2}{p}\hat{J}\big)\bigg),mQ-\frac{1}{i\beta k}\frac{\de_t p}{p}mZ}\right|.
\end{align}
The errors for the higher--order terms are
\begin{align}
\label{def:Tho}
\sfT_{\sfho}:=&\left|\jap{m\cF\left(\nabla^\perp \Psi\cdot \nabla \Omega\right),m\Omega}\right|+\left|\jap{m\cF\left(\nabla^\perp \Phi\cdot \nabla J\right),m\Omega}\right|\\
&+\left|\jap{m\cF\left(\nabla^\perp \Psi\cdot \nabla J\right),mJ}\right|+\left|\jap{m\cF\left(\nabla^\perp \Phi\cdot \nabla \Omega\right),mJ}\right|.
\end{align}
%%Commutators if needed
%\begin{align}
%\label{def:Tsym}
%\sfT_{\sfsym}:=&\left|\jap{\cF\left(\left[w,\nabla^\perp \Psi\right]\cdot \nabla \Omega\right),mZ}\right|+\left|\jap{\cF\left(\left[w,\nabla^\perp \Phi\right]\cdot \nabla J\right),mZ}\right|\\
%&\left|\jap{\cF\left(\left[w,\nabla^\perp \Psi\right]\cdot \nabla J\right),mQ}\right|+\left|\jap{\cF\left(\left[w,\nabla^\perp \Phi\right]\cdot \nabla \Omega\right),mQ}\right|\\
%&\frac{1}{\beta}\left|\jap{\cF\left(\left[\de_X^{-1}\frac{\de_t p}{p}w,\nabla^\perp \Psi\right]\cdot \nabla \Omega\right),mQ}\right|+\frac{1}{\beta}\left|\jap{\cF\left(\left[\de_X^{-1}\frac{\de_t p}{p}w,\nabla^\perp \Phi\right]\cdot \nabla J\right),mZ}\right|\\
%&\frac{1}{\beta}\left|\jap{\cF\left(\left[\de_X^{-1}\frac{\de_t p}{p}w,\nabla^\perp \Psi\right]\cdot \nabla J\right),mQ}\right|+\frac{1}{\beta}\left|\jap{\cF\left(\left[\de_X^{-1}\frac{\de_t p}{p}w,\nabla^\perp \Phi\right]\cdot \nabla \Omega\right),mQ}\right|
%\end{align}
and 
\begin{align}
\label{def:Sho}
\sfS_{\sfho}:=\,&\left|\jap{m\bigg(\frac{\de_t p}{p}\hat{J}*\big(\hat{\Omega}-2\frac{\ell^2}{p}\hat{\Omega}\big)\bigg),mJ}\right|+\left|\jap{m\bigg(\frac{\de_t p}{p}\hat{\Omega}*\big(\hat{J}-2\frac{\ell^2}{p}\hat{J}\big)\bigg),mJ}\right|.
\end{align}
The error term for the zero-mode functional is 
\begin{align}
\label{def:Rneq}
\mathsf{R}_{\neq}:=&\left|\jap{\jap{\de_Y}^N\left(U^2_{\neq} U^1_{\neq}\right)_0,\jap{\de_Y}^N(\de_YU_0^1)}\right|+\left|\jap{\jap{\de_Y}^N\left(B^2_{\neq} B^1_{\neq}\right)_0,\jap{\de_Y}^N(\de_YU_0^1)}\right|\\
&\left|\jap{\jap{\de_Y}^N\left(U^2_{\neq} B^1_{\neq}\right)_0,\jap{\de_Y}^N(\de_YB_0^1)}\right|+\left|\jap{\jap{\de_Y}^N\left(B^2_{\neq} U^1_{\neq}\right)_0,\jap{\de_Y}^N(\de_YB_0^1)}\right|\\
&+\frac{1}{\jap{t}^2}\bigg(\left|\jap{\jap{\de_Y}^N\left(U^2_{\neq} \Omega_{\neq}\right)_0,\jap{\de_Y}^N(\de_Y\Omega_0^1)}\right|+\left|\jap{\jap{\de_Y}^N\left(B^2_{\neq} J_{\neq}\right)_0,\jap{\de_Y}^N(\de_Y\Omega_0^1)}\right|\\
&\qquad\qquad +\left|\jap{\jap{\de_Y}^N\left(U^2_{\neq} J_{\neq}\right)_0,\jap{\de_Y}^N(\de_YJ_0^1)}\right|+\left|\jap{\jap{\de_Y}^N\left(B^2_{\neq} \Omega_{\neq}\right)_0,\jap{\de_Y}^N(\de_YJ_0^1)}\right|\\
&\qquad\qquad+\left|\jap{\jap{\eta}^N\bigg(2\de_XB^1_{\neq}\big(\Omega_{\neq}-2\de_{XX}\Delta_L^{-1}J_{\neq}\big)\bigg)_0,\jap{\eta}^N\hat{J}_0}\right|\\
&\qquad\qquad+\left|\jap{\jap{\eta}^N\bigg((2\de_XU^1_{\neq}\big(J_{\neq}-2\de_{XX}\Delta_L^{-1}J_{\neq}\big)\bigg)_0,\jap{\eta}^N\hat{J}_0}\right|\bigg)
.
\end{align}
\end{lemma}

\begin{remark}
In the transport error term we could easily introduce commutators by exploiting the divergence free condition on $U$ and $B$. However, since we are not worried of loss of derivatives thanks to the dissipation, we will see that commutators are not necessary to close the argument with the threshold $\nu^\frac23$, which is expected as explained in Remark \ref{rem:threshold}.
\end{remark}
\begin{proof}
In the proof, we omit the subscript $k$ for simplicity of notation. First of all, taking the Fourier transform of \eqref{def:OJsystem}, we compute that
\begin{align}
&\de_tZ=-\left(\nu p-\frac{\de_t m}{m}\right)Z-\frac12\frac{\de_t p}{p}Z+\beta ik Q+m\sqrt{\frac{k^2}{p}}\cF(\mathrm{NL}_{\Omega}),\\
&\de_tQ=- \left(\nu p-\frac{\de_t m}{m}\right)Q+\frac12\frac{\de_t p}{p}Q+\beta i kZ+m\sqrt{\frac{k^2}{p}}\cF(\mathrm{NL}_{J}).
\end{align}
Therefore, when computing the time-derivative of $\sfE_{\sfsym}$ we readily see that the contributions from the linear part of the equations are controlled as in 
the proof of Proposition \ref{prop:keylin} and give us the left-hand side of \eqref{bd:dtEsym}. Indeed, all the linear estimates are valid point-wise in frequency and here we are just integrating in space. The definition of the nonlinear terms follows by triangle inequality and Plancherel's theorem, whence proving \eqref{bd:dtEsym}.

To prove \eqref{bd:dtEho}, by \eqref{def:OJsystem} we get 
\begin{align}
\label{eq:dtEho}\frac{\dd }{\dd t}\sfE_{\sfho}+\sfD_{\sfho}=\,&\delta_0\nu^\frac13(\norm{m\Omega}^2+\norm{mJ}^2)\\
\label{eq:antiEho}&+\beta(\jap{\de_X mJ,m\Omega}+\jap{\de_X m\Omega,mJ})\\
\label{eq:symmEho}&+\jap{\frac{\de_t p}{p}m\hat{J},m\hat{J}}\\
&+\jap{m(\mathrm{NL}_{\Omega}),m\Omega}+\jap{m(\mathrm{NL}_{J}),mJ}.
\end{align}
Appealing to \eqref{bd:mnu}, we have
\begin{align}
\delta_0\nu^\frac13\norm{m(\Omega,J)}^2\leq 4\delta_0\sfD_{\sfho},
\end{align}
where we used $\mu\geq \nu$. Hence, for $\delta_0$ sufficiently we can absorb the term on the right-hand side of \eqref{eq:dtEho} to the left hand-side and remain with $\sfD_{\sfho}/16$ as in \eqref{bd:dtEho}. The term in \eqref{eq:antiEho} is clearly zero. For the term in \eqref{eq:symmEho}, reasoning as done in \eqref{bd:linOmJ}, we get 
\begin{equation}
\left|\jap{\frac{\de_t p}{p}m\hat{J},m\hat{J}}\right|\leq2 |\langle mQ,m\hat{J}\rangle|\leq 4\sqrt{\sfE_{\sfsym}}\sqrt{\sfE_{\sfho}}.
\end{equation}
For the nonlinear terms we only apply the triangle inequality.

It remains to compute the errors for the zero modes. We first write down the equations for the $x$-average of the velocity and magnetic fields. Since both $U$ and $B$ are divergence free, we have $U_0^2=B_0^2=0$. Hence, it is not difficult to check that the equations of $(U_0^1,B_0^1)$ are given by (see for instance \cite[eq. (2.11)]{zhao2023asymptotic})
\begin{align}
\label{eq:U01}&\de_t U_0^1-\nu \de_{YY}U_0^1=-(\nabla^\perp\Psi_\neq\cdot \nabla U_\neq^1)_0+(\nabla^\perp\Phi_\neq\cdot \nabla B^1_\neq)_0\\
\label{eq:B01}&\de_t B_0^1-\mu \de_{YY}B_0^1=-(\nabla^\perp\Psi_\neq\cdot \nabla B^1_\neq)_0+(\nabla^\perp\Phi_\neq\cdot \nabla U^1_\neq)_0,
\end{align}
where we also used the identity
\begin{equation}
(F G)_0=( F_{\neq} G_{\neq})_0.
\end{equation}
 The equations for $(\Omega_0,J_0)$ are like  \eqref{eq:U01}-\eqref{eq:B01} with the changes $(U_0^1,B_0^1)\to (\Omega_0,J_0)$, $(U_{\neq},B_{\neq})\to (\Omega_\neq,J_\neq)$ and the $x$-average of the stretching terms of $\mathrm{NL}_J$ in \eqref{def:NLJS}. To prove that $\sfR_{\neq}$ only involves some specific components of the nonlinearity, we observe the following general cancellations: for any multiplier $q$ and functions $F,G,H$, after a few integration by parts we obtain
 \begin{align}
& \jap{q(\nabla^\perp F_{\neq}\cdot \nabla G_{\neq})_0,qH_0}= -\jap{q(\de_Y F_{\neq}\de_X G_{\neq}),qH_0}+\jap{q(\de_X F_{\neq}\de_Y G_{\neq}),qH_0}\\
&\qquad = \jap{q((\de_XF_{\neq})G_{\neq}),q\de_Y H_0}.
 \end{align}
 It is then enough to recall that $\de_X(\Psi_{\neq},\Phi_{\neq})=(U^2_{\neq},B^2_{\neq})$ to obtain all the transport type terms in $\sfR_{\neq}$. For the stretching nonlinearity in \eqref{def:NLJS}, we use $(\de_Y-t\de_X)(\Psi_{\neq},\Phi_{\neq})=-(U^1_{\neq},B^1_{\neq})$ to conclude the proof of the lemma.
%It remains to compute the nonlinear terms. From the nonlinearity $\mathrm{NL}_{J}$ , we readily see why we have the stretching term \eqref{def:Ssym}. For the remaining transport terms, we observe the following cancellations: for any Fourier multiplier $w$ and functions $F,G,H$ we have
%\begin{align}
%\label{bd:comm1}
%&\jap{w(\nabla^\perp G\cdot \nabla F),wF}=\jap{[w,\nabla^\perp G]\cdot \nabla F),wF},\\
%\label{bd:comm2}&\jap{w(\nabla^\perp G\cdot \nabla F),wH}+\jap{w(\nabla^\perp G\cdot \nabla H),wF}\\
%\notag&\qquad \qquad =\jap{[w,\nabla^\perp G]\cdot \nabla F),wH}+\jap{[w,\nabla^\perp G]\cdot \nabla H),wF},
%\end{align}
%where we simply used the antysymmetry of the operator $\nabla^\perp G\cdot \nabla$ (or equivalently $\nabla\cdot \nabla^\perp G=0$). Hence, looking at \eqref{def:NLO}-\eqref{def:NLJ}, from \eqref{def:w} we see that we can always introduce the commutators between the weight $w$ and $\nabla^\perp G$ with $G\in \{\Psi,\Phi\}$. For instance, when doing the time derivative of $\norm{mZ}^2+\norm{mQ}^2$ we get the following terms 
%\begin{align}
%&-\jap{w(\nabla^\perp \Psi\cdot \nabla\Omega),mZ}-\jap{w(\nabla^\perp \Psi\cdot \nabla J),mQ}\\
%&+\jap{w(\nabla^\perp \Phi\cdot \nabla J ),mZ}+\jap{w(\nabla^\perp \Phi\cdot \nabla\Omega),mQ}.
%\end{align}
%Combining \eqref{eq:id} and \eqref{bd:comm1}-\eqref{bd:comm2} we deduce why we can get the commutators.
\end{proof}

With the energy identities at hand, we are ready to set up the bootstrap argument. First, we assume the following.
\medskip

\begin{quote}
\textbf{Bootstrap hypothesis:} Assume that there exists $T_\star\geq 1$ such that for all $1/2\leq t\leq T_\star$ the following inequalities holds true:
\begin{align}
\label{Bsym}\tag{$\mathrm{H}_{\sfsym}$} &\sfE_{\sfsym}(t)+\frac{1}{16}\int_0^t\sfD_{\sfsym}(\tau)\dd \tau\leq 10\eps^2,\\
\label{Bho}\tag{$\mathrm{H}_{\sfho}$}& \sfE_{\sfho}(t)+\frac{1}{16}\int_0^t\sfD_{\sfho}(\tau)\dd \tau\leq C_1\eps^2\jap{t}^2,\\
\label{B0}\tag{$\mathrm{H}_{0}$} &\sfE_{0}(t)+\frac{1}{16}\int_0^t\sfD_0(\tau)\dd \tau\leq 100\eps^2,
\end{align}
with $C_1=4000$.
\end{quote}

\medskip 
By a standard local well-posedness argument (which can be obtained from the bounds in Lemma \ref{lem:enid}), we know that for $\eps_0$ sufficiently small the hypothesis \eqref{Bsym}-\eqref{B0} holds true with $T_\star=1$ and all the constants on the right-hand side divided by $4$. Then, we aim at improving the bounds \eqref{Bsym}-\eqref{B0} so that, by continuity and the fact that the interval $[1/2,T_\star]$ will be open, closed and connected, we get $T_\star=+\infty$. In particular, our goal is to prove the following.
\begin{proposition}[Bootstrap improvement]
\label{prop:boot}
Under the hypothesis of Theorem \ref{th:main}, there exists $0<\eps_0=\eps_0(N,\beta)<1/2$  with the following property. If $\eps<\eps_0$ and \eqref{Bsym}-\eqref{B0} hold on $[1/2,T_\star]$, then for any $t\in[1/2,T_*]$ the estimates \eqref{Bsym}-\eqref{B0} are true with all the constants on the right-hand side of \eqref{Bsym}-\eqref{B0} divided by a factor $2$.
\end{proposition}

From this proposition, which we prove in the next section, the proof of Theorem \ref{th:main} readily follows by the definition of the energies and the bounds \eqref{bd:U1triv}-\eqref{bd:U2triv}. 
\section{Proof of the bootstrap proposition}
\label{sec:NLproof}
This section is dedicated to the proof Proposition \ref{prop:boot}, which implies Theorem \ref{th:main}, and constitutes the  core of this paper. To improve the bounds \eqref{Bsym}-\eqref{B0}, we  need to introduce some  useful technical results.
\subsection{Toolbox}
We introduce the \textit{resonant intervals} as follows: we say that 
\begin{equation}
\label{def:resonant}
t\in I_{k,\eta} \qquad \text{ if } \qquad \left|t-\frac{\eta}{k}\right|\leq \frac{|\eta|}{2k^2}.
\end{equation}
These intervals are usually defined in a slightly more precise way in inviscid problems, e.g. \cite{BM15}. This definition is sufficient for us since we never have to define weights using the resonant intervals. In fact, we only need them for notational purposes when splitting integrals.

We recall the following properties of the weight $p_k(t,\eta)$.
\begin{lemma}
\label{lem:commp}
For any $t,k,\eta,\ell,\xi$, the following inequalities holds true 
\begin{align}
\label{bd:pp}
\sqrt{\frac{p_\ell(\xi)}{p_k(\eta)}}\leq \jap{|k-\ell,\eta-\xi|}^3\begin{cases}\displaystyle \frac{|\eta|}{k^2(1+|\frac{\eta}{k}-t|)}, &\qquad \text{if  } \, t\in I_{k,\eta}\cap I^c_{\ell,\xi}\\
1 &\qquad \text{otherwise}
\end{cases}
\end{align}
When $k=\ell$ we have the improved estimate 
\begin{equation}
\label{eq:pkp}
\sqrt{\frac{p_k(\xi)}{p_k(\eta)}}\leq 1+\frac{|\eta-\xi|}{|k|(1+|\frac{\eta}{k}-t|)}.
\end{equation}
\end{lemma}

\begin{proof}
This lemma is a version of \cite[Lemma 4.14]{BBCZD23}, where \eqref{bd:pp} is proved.  The bound \eqref{eq:pkp} is in the proof of \cite[Lemma 4.14]{BBCZD23} as well, which follows by 
\begin{align}
\sqrt{\frac{p_k(\xi)}{p_k(\eta)}}=\frac{1+|\frac{\xi}{k}-t|}{1+|\frac{\eta}{k}-t|}\leq 1+\frac{|(\frac{\xi}{k}-\frac{\eta}{k})+(\frac{\eta}{k}-t)|-|\frac{\eta}{k}-t|}{1+|\frac{\eta}{k}-t|}\leq 1+\frac{|\eta-\xi|}{|k|(1+|\frac{\eta}{k}-t|)}.
\end{align}
\end{proof}

The following \textit{lossy elliptic estimate} enables us to exploit the invisicid damping by paying regularity.
\begin{lemma}
\label{lem:elliptic}
For any $s\geq0$
\begin{equation}
\norm{(-\Delta_L)^{-1}F}_{H^{s}}\lesssim \frac{1}{\jap{t}^2}\norm{F}_{H^{s+2}}.
\end{equation}
\end{lemma}
The proof of this lemma is an application of the inequality $\jap{a-b}\jap{b}\gtrsim \jap{a}$, see \cite{BM15}.
We also record the following bounds that follows directly by the definition of $m^d$, see \eqref{def:md},
\begin{equation}
\label{bd:pneq0}
\sqrt{\frac{k^2}{p_k(t,\eta)}}=\frac{1}{\sqrt{C_\beta}}\sqrt{\frac{\de_tm^d_k(t,\eta)}{m^d_k(t,\eta)}}, \qquad \frac{|k|}{p_k(t,\eta)}\lesssim\sqrt{\frac{\de_tm^d_k(t,\eta)}{m^d_k(t,\eta)}}\sqrt{\frac{k^2}{p_k(t,\eta)}}.
\end{equation}

\subsection{Bounds on the symmetric variables}
In this section, we aim at proving that \eqref{Bsym} holds true with $10$ replaced by $5$. Looking at the energy identity \eqref{bd:dtEsym} and the definition of $\sfT_{\sfsym}$ \eqref{def:Tsym} and $\sfS_{\sfsym}$ \eqref{def:Ssym}, we see that all the nonlinear error terms are of the following type:
\begin{align}
\label{def:cTsym}
&\cT_{\sfsym}(F,G,H):=\left|\jap{\sqrt{\frac{k^2}{p}}m\cF\left(\nabla^\perp\Delta^{-1}_L F\cdot \nabla G \right), \hat{H}}\right|\\
\label{def:cSym}&\cS_{\sfsym}(F,G,H):=\left|\jap{\sqrt{\frac{k^2}{p}}m\bigg(\frac{\de_t p}{p}\hat{F}*\big(\hat{G}-2\frac{\ell^2}{p}\hat{G}\big)\bigg),H}\right|.
\end{align}
Moreover, in terms of bound to perform, thanks to the definition of the functionals \eqref{def:Esym}--\eqref{def:D0} and the bootstrap hypotheses \eqref{Bsym}--\eqref{B0}, we see that there is actually no difference between $(\Psi,\Omega)$ and $(\Phi,J)$. In the next lemma we collect the bounds we need for the transport and stretching nonlinearities respectively.
%
%We denote 
%\begin{equation}
%\label{def:Tsym}
%\cT_{\sfsym}(F_{\neq},G_{\neq},H_{\neq}):=\left|\jap{\sqrt{\frac{k^2}{p}}m\cF\left(\nabla^\perp\Delta^{-1}_L F_{\neq}\cdot \nabla G_{\neq} \right), \hat{H}_{\neq}}\right|.
%\end{equation}
\begin{lemma}
\label{lem:keyEsym}
Let $m$ be the Fourier multiplier defined in \eqref{def:m} with $N>10$. The following inequalities holds true: 
\begin{align}
\label{bd:Tsimneqneq}
\cT_{\sfsym}(F_{\neq},G_{\neq},H)\lesssim\, &\e^{-\delta_0\nu^\frac13 t}\norm{mF_{\neq}}_{L^2}\norm{\sqrt{\frac{k^2}{p}}mG_{\neq}}_{L^2}\norm{\sqrt{\frac{\de_t m^d}{m^d}}H}_{L^2}\\
&\,+\frac{1}{\jap{t}}\e^{-\delta_0\nu^\frac13 t}\norm{mF_{\neq}}_{L^2}\norm{\sqrt{\frac{k^2}{p}}m|\nabla_L|G_{\neq}}_{L^2}\norm{H}_{L^2}\\
&\,+ \jap{t} \e^{-\delta_0\nu^\frac13 t}\norm{mG_{\neq}}_{L^2}\norm{\sqrt{\frac{\de_t m^d}{m^d}}\sqrt{\frac{k^2}{p}}mF_{\neq}}_{L^2}\norm{\sqrt{\frac{\de_t m^d}{m^d}}H}_{L^2}\\
&\, +\e^{-\delta_0\nu^\frac13 t}\norm{mG_{\neq}}_{L^2}\norm{\sqrt{\frac{k^2}{p}}mF_{\neq}}_{L^2}\norm{\sqrt{\frac{\de_t m^d}{m^d}}H}_{L^2}.
\end{align}
Denoting $(\nabla^\perp\Delta_L^{-1}F)_0=(V_{F,0}^1,0)$, one has
\begin{align}
\label{bd:Tsim0neq}
\cT_{\sfsym}(F_{0},G_{\neq},H)\lesssim\,& \norm{V_{F,0}^1}_{H^N}\norm{\sqrt{\frac{k^2}{p}}m\de_X G_{\neq}}_{L^2}\norm{H}_{L^2}\\
&+\norm{\de_YV_{F,0}^1}_{H^N}\norm{\sqrt{\frac{k^2}{p}}m G_{\neq}}_{L^2}\norm{\sqrt{\frac{\de_t m^d}{m^d}}H}_{L^2}.
\end{align}
Moreover 
\begin{align}
\label{bd:Tsimneq0}
\cT_{\sfsym}(F_{\neq},G_{0},H)\lesssim\, &\norm{G_0}_{H^{3}}\norm{\sqrt{\frac{\de_tm^d}{m^d}}\sqrt{\frac{k^2}{p}}mF_{\neq}}_{L^2}\norm{\sqrt{\frac{\de_t m^d}{m^d}}H}_{L^2}\\
&+\frac{1}{\jap{t}^2}\norm{mF_{\neq}}_{L^2}\norm{\de_YG_0}_{H^{N}}\norm{\sqrt{\frac{\de_tm^d}{m^d}}H}_{L^2}.
\end{align}
For the stretching nonlinearities we have the following
\begin{align}
\label{bd:Ssimneqneq}
&\cS_{\sfsym}(F,G_{\neq},H)\lesssim\, \e^{-\delta_0\nu^\frac13t}\norm{\sqrt{\frac{k^2}{p}}mF_{\neq}}_{L^2}\norm{mG_{\neq}}_{L^2}\norm{\sqrt{\frac{\de_t m^d}{m^d}}H}_{L^2},\\
\label{bd:Ssimneq0}
&\cS_{\sfsym}(F,G_{0},H)\lesssim\,\left(\norm{\sqrt{\frac{k^2}{p}}mF_{\neq}}_{L^2}\norm{G_{0}}_{H^3} + \frac{1}{\jap{t}}\norm{mF_{\neq}}_{L^2}\norm{G_{0}}_{H^N}\right)\norm{\sqrt{\frac{\de_t m^d}{m^d}}H}_{L^2}.
\end{align}
\end{lemma}
\begin{remark}
The bounds on $\cT_{\sfsym}(F_0,G_\neq,H)$ are not optimal since we could exploit commutators to avoid losing an $x$-derivative on $G_{\neq}$. However, for the threshold $\eps\ll \nu^{\frac23}$ this does not seem necessary.
\end{remark}
Before proving the key lemma above, we first show how to improve the bootstrap hypothesis \eqref{Bsym} with the estimates in Lemma \ref{lem:keyEsym}.\begin{proof}[Proof: improvement of \eqref{Bsym}]
Since $|\de_tp|/p\leq 1$, in the nonlinear term $\sfT_{\sfsym}$ (see \eqref{def:Tsym}) we can just study, for example, the term 
\begin{equation}
\cT_{\sfsym}(\Omega,\Omega,mZ)\leq \cT_{\sfsym}(\Omega_{\neq},\Omega_{\neq},mZ) +\cT_{\sfsym}(\Omega_0,\Omega_{\neq},mZ)+\cT_{\sfsym}(\Omega_{\neq},\Omega_{0},mZ)
\end{equation}
where we used that $\Delta_L^{-1}\Omega=\Psi$. Recalling the definition of $Z$ given in \eqref{def:ZQ}, applying \eqref{bd:Tsimneqneq} we deduce that 
\begin{align}
\cT_{\sfsym}(\Omega_{\neq},\Omega_{\neq},mZ)\lesssim \,&\e^{-\delta_0\nu^\frac13 t}\norm{m\Omega_{\neq}}_{L^2}\norm{mZ}_{L^2}\norm{\sqrt{\frac{\de_t m^d}{m^d}}mZ}_{L^2}\\
&\,+\frac{1}{\jap{t}}\e^{-\delta_0\nu^\frac13 t}\norm{m\Omega_{\neq}}_{L^2}\norm{|\nabla_L|mZ}_{L^2}\norm{mZ}_{L^2}\\
&\,+ \jap{t} \e^{-\delta_0\nu^\frac13 t}\norm{m\Omega_{\neq}}_{L^2}\norm{\sqrt{\frac{\de_t m^d}{m^d}}mZ}_{L^2}\norm{\sqrt{\frac{\de_t m^d}{m^d}}mZ}_{L^2}
%&\, +\e^{-\delta_0\nu^\frac13 t}\norm{m\Omega_{\neq}}_{L^2}\norm{mZ}_{L^2}\norm{\sqrt{\frac{\de_t m^d}{m^d}}mZ}_{L^2}.
\end{align}
From the definitions of $\sfE_{\sfsym},\sfE_{\sfho}$ and $\sfD_{\sfsym}$, see respectively \eqref{def:Esym}, \eqref{def:Eho} and \eqref{def:Dsym}, we rewrite this bound as
\begin{align}
\cT_{\sfsym}(\Omega_{\neq},\Omega_{\neq},mZ)\lesssim \e^{-\delta_0\nu^\frac13 t}\sqrt{\sfE_{\sfho}}\sqrt{\sfD_{\sfsym}}\left(\sqrt{\sfE_{\sfsym}}+\frac{1}{\jap{t}}\nu^{-\frac12}\sqrt{\sfE_{\sfsym}}+\jap{t}\sqrt{\sfD_{\sfsym}}\right).
\end{align}
Appealing to the boostrap hypothesis \eqref{Bsym}--\eqref{Bho}, we get 
\begin{align}
\cT_{\sfsym}(\Omega_{\neq},\Omega_{\neq},mZ)&\lesssim \e^{-\delta_0\nu^\frac13 t}\eps\jap{t}\sqrt{\sfD_{\sfsym}}\left(\eps+\frac{1}{\jap{t}}\nu^{-\frac12}\eps+\jap{t}\sqrt{\sfD_{\sfsym}}\right)\\
&\lesssim(\eps^2\nu^{-\frac13}+\eps^2\nu^{-\frac12})\e^{-\delta_0\nu^\frac13 t/2}\sqrt{\sfD_{\sfsym}}+\eps\nu^{-\frac23}\sfD_{\sfsym}\\
&\lesssim \eps\nu^{-\frac23}\sfD_{\sfsym}+\eps^2(\eps\nu^{-\frac23})\nu^\frac13\e^{-\delta_0\nu^\frac13 t}.
\end{align}
Integrating in time and using the bootstrap hypotheses, we have
\begin{align}
\label{bd:cTB1}
\int_0^t\cT_{\sfsym}(\Omega_{\neq},\Omega_{\neq},mZ)\dd \tau\dd \tau\lesssim (\eps\nu^{-\frac23}) \eps^2.
\end{align}
Since $\nabla^\perp\Delta_L^{-1}\Omega_0=U_0^1$ and $|\de_X|\leq |\nabla_L|$, from \eqref{bd:Tsim0neq} we get 
\begin{align}
\cT_{\sfsym}(\Omega_{0},\Omega_{\neq},mZ)\lesssim\,& \norm{U_{0}^1}_{H^N}\norm{m\de_X Z}_{L^2}\norm{mZ}_{L^2}+\norm{\de_YU_{0}^1}_{H^N}\norm{m Z}_{L^2}\norm{\sqrt{\frac{\de_t m^d}{m^d}}mZ}_{L^2}\\
\lesssim \, & \nu^{-\frac12}\sqrt{\sfE_{0}}\sqrt{\sfE_{\sfsym}}\sqrt{\sfD_{\sfsym}}+\nu^{-\frac12}\sqrt{\sfD_{0}}\sqrt{\sfE_{\sfsym}}\sqrt{\sfD_{\sfsym}}.
\end{align}
From the property \eqref{eq:keymnu}, we know that 
\begin{equation}
\label{bd:ED}
\sqrt{\sfE_{\sfsym}}\lesssim \nu^{-\frac16}\sqrt{\sfD_{\sfsym}}.
\end{equation}
Using the bootstrap hypotheses we then deduce 
\begin{equation}
\label{bd:cTB2}
\int_0^t\cT_{\sfsym}(\Omega_{0},\Omega_{\neq},mZ)\lesssim \eps\nu^{-\frac23}\int_0^t\sfD_{\sfsym}\dd \tau+\eps\nu^{-\frac12}\int_0^t\sfD_0\dd \tau\lesssim  (\eps\nu^{-\frac23})\eps^2.
\end{equation}

For the last term of the transport nonlinearity, since $\norm{\Omega_0}_{H^3}\lesssim \norm{U_0^1}_{H^4}\lesssim \sqrt{\sfE_0}$, applying \eqref{bd:Tsimneq0} we have
\begin{align}
\cT_{\sfsym}(\Omega_{\neq},\Omega_{0},mZ)\lesssim\, &\norm{\Omega_0}_{H^{3}}\norm{\sqrt{\frac{\de_tm^d}{m^d}}mZ}_{L^2}\norm{\sqrt{\frac{\de_t m^d}{m^d}}mZ}_{L^2}\\
&+\frac{1}{\jap{t}^2}\norm{m\Omega_{\neq}}_{L^2}\norm{\de_Y\Omega_0}_{H^{N}}\norm{\sqrt{\frac{\de_tm^d}{m^d}}mZ}_{L^2}\\
\lesssim\, &\sqrt{\sfE_0}\sfD_{\sfsym}+\frac{\nu^{-\frac12}}{\jap{t}}\sqrt{\sfE_{\sfho}}\sqrt{\sfD_0}\sqrt{\sfD_{\sfsym}}.
\end{align}
Using the bootstrap assumptions we deduce 
\begin{align}
\label{bd:cTB3}
\int_0^t\cT_{\sfsym}(\Omega_{\neq},\Omega_{0},mZ)\dd \tau\lesssim\, (\eps\nu^{-\frac12})\int_0^t(\sfD_{\sfsym}+\sfD_0)\dd \tau\lesssim (\eps\nu^{-\frac12})\eps^2  \end{align}
The structure of all the other transport nonlinearities enables us to apply the exact same procedure to the term we have just controlled. Therefore, from the bounds \eqref{bd:cTB1}, \eqref{bd:cTB2} and \eqref{bd:cTB3} we conclude that
\begin{align}
\label{bd:Tsymfinal}
\int_0^t\sfT_{\sfsym}\dd \tau\lesssim (\eps\nu^{-\frac23})\eps^2.
\end{align}

Turning our attention to the stretching nonlinearities, we can again explicitly handle just one of them, say 
\begin{equation}
\cS_{\sfsym}(J,\Omega,mQ)\leq \cS_{\sfsym}(J,\Omega_{\neq},mQ)+\cS_{\sfsym}(J,\Omega_0,mQ).
\end{equation}
From \eqref{bd:Ssimneqneq} we get 
\begin{align}
\cS_{\sfsym}(J,\Omega_{\neq},mQ)&\lesssim\, \e^{-\delta_0\nu^\frac13t}\norm{mQ}_{L^2}\norm{m\Omega_{\neq}}_{L^2}\norm{\sqrt{\frac{\de_t m^d}{m^d}}mQ}_{L^2}\\
&\lesssim \e^{-\delta_0\nu^\frac13t}\sqrt{\sfE_{\sfsym}}\sqrt{\sfE_{\sfho}}\sqrt{\sfD_{\sfsym}}\lesssim \eps^2\jap{t}\e^{-\delta_0\nu^\frac13t}\sqrt{\sfD_{\sfsym}}
\end{align}
where we used the bootstrap hypotheses. Therefore, 
\begin{align}
\label{bd:SB1}
\int_0^t\cS_{\sfsym}(J,\Omega_{\neq},mQ)\dd \tau\lesssim (\eps \nu^{-\frac12})\int_0^t(\sfD_{\sfsym}+\eps^2\nu^{\frac13}\e^{-\delta_0\nu^\frac13\tau})\dd \tau\lesssim  (\eps \nu^{-\frac12})\eps^2
\end{align}
Looking at \eqref{bd:Ssimneq0}, we get
\begin{align}
\cS_{\sfsym}(J,\Omega_{0},mQ)&\lesssim\,\left(\norm{mQ}_{L^2}\norm{\Omega_{0}}_{H^3} + \frac{1}{\jap{t}}\norm{mJ_{\neq}}_{L^2}\norm{\Omega_{0}}_{H^N}\right)\norm{\sqrt{\frac{\de_t m^d}{m^d}}mQ}_{L^2}.
\end{align}
Now we observe that 
\begin{align}
\norm{\Omega_{0}}_{H^N}=\norm{\de_YU_0^1}_{H^N}\lesssim \nu^{-\frac12}\sqrt{\sfD_{0}}.
\end{align}
Thus, using again \eqref{bd:ED}, we have
\begin{align}
\label{bd:SB2}
\cS_{\sfsym}(J,\Omega_{0},mQ)&\lesssim \sqrt{\sfD_{\sfsym}}(\sqrt{\sfE_{\sfsym}}\sqrt{\sfE_{0}}+\frac{\nu^{-\frac12}}{\jap{t}}\sqrt{\sfE_{\sfho}}\sqrt{\sfD_{0}})\\
&\lesssim (\eps\nu^{-\frac16}+\eps\nu^{-\frac12})\sfD_{\sfsym}+\eps \nu^{-\frac12}\sfD_0
\end{align}
Arguing similarly for the other stretching term, combining \eqref{bd:SB1} with \eqref{bd:SB2} we get
\begin{align}
\label{bd:Ssymfinal}
\int_0^t\sfS_{\sfsym}\dd \tau\lesssim (\eps\nu^{-\frac12})\eps^2
\end{align}
Finally, using the bound above and \eqref{bd:Tsymfinal}, integrating in time the energy inequality \eqref{bd:dtEsym} we get 
\begin{align}
\sfE_{\sfsym}+\frac{1}{16}\int_0^t\sfD_{\sfsym}\dd \tau\leq \eps^2 +C(\eps \nu^{-\frac23})\eps^2,
\end{align}
where $C=C(N,\delta_0,\beta)>1$. By choosing $\eps_0\ll \nu^\frac23$, we improve the bound \eqref{Bsym} and conclude the proof.
\end{proof}
 
 We finally present the proof of Lemma \ref{lem:keyEsym}.
\begin{proof}[Proof of Lemma \ref{lem:keyEsym}]
We split the proof for each of the bounds \eqref{bd:Tsimneqneq}, \eqref{bd:Tsim0neq}, \eqref{bd:Tsimneq0}, \eqref{bd:Ssimneqneq} and \eqref{bd:Ssimneq0}.

\medskip 

$\bullet$ \textit{Proof of \eqref{bd:Tsimneqneq}:} appealing to the paraproduct decomposition \eqref{def:LH}, we see that 
\begin{align}
\cT_{\sfsym}(F_{\neq},G_{\neq},H)\leq \cT_{\sfsym}(F_{\neq}^{Lo},G_{\neq}^{Hi},H)+\cT_{\sfsym}(F_{\neq}^{Hi},G_{\neq}^{Lo},H)
\end{align}
We study separately the low-high and the high-low terms.
\medskip

\noindent  $\square$ \textit{Control of the low-high term.} 
We split the integral to handle separately the resonant and non-resonant case ($t\in I_{k,\eta}\cap I_{\ell,\xi}^c$ or not), that is 
 \begin{equation}
\cT_{\sfsym}(F_{\neq}^{Lo},G_{\neq}^{Hi},H)\leq \cT^R_{\sfsym}(F_{\neq}^{Lo},G_{\neq}^{Hi},H_{\neq})+\cT^{NR}_{\sfsym}(F_{\neq}^{Lo},G_{\neq}^{Hi},H),
 \end{equation}
 where we define 
 \begin{align}
 \notag \cT^R_{\sfsym}(F_{\neq}^{Lo},G_{\neq}^{Hi},H):=\sum_{k,\ell \in \mathbb{Z}}\iint_{\RR^2}&\mathbbm{1}_{\{t\in I_{k,\eta}\cap I_{\ell,\xi}^c\}}\sqrt{\frac{k^2}{p_k(t,\eta)}}m_k(t,\eta)\frac{|k-\ell,\eta-\xi|}{p_{k-\ell}(t,\eta-\xi)}|\hat{F}^{Lo}_{\neq}|_{k-\ell}(\eta-\xi)\\
 \label{def:TRsym}&\times |\ell,\xi| |\hat{G}^{Hi}|_{\ell}(\xi)|\hat{H}|_{k}(\eta) \dd \eta\dd \xi\\
\notag  \cT^{NR}_{\sfsym}(F_{\neq}^{Lo},G_{\neq}^{Hi},H):=\sum_{k,\ell \in \mathbb{Z}}\iint_{\RR^2}&\mathbbm{1}_{\{t\notin I_{k,\eta}\cap I_{\ell,\xi}^c\}}\sqrt{\frac{k^2}{p_k(t,\eta)}}m_k(t,\eta)\frac{|k-\ell,\eta-\xi|}{p_{k-\ell}(t,\eta-\xi)}|\hat{F}^{Lo}_{\neq}|_{k-\ell}(\eta-\xi)\\
 \label{def:TNRsym}&\times |\ell,\xi| |\hat{G}^{Hi}|_{\ell}(\xi)|\hat{H}|_{k}(\eta) \dd \eta\dd \xi.
 \end{align}
 By definition of the paraproduct \eqref{def:LH} we know that $|k,\eta|\leq 3|\ell,\xi|$. Hence, since $m^{d},m^{\nu},m^s$ are uniformly bounded Fourier multipliers, we deduce that 
 \begin{equation}
 \label{bd:mk}
 m_k(t,\eta)\lesssim m_\ell(t,\xi)
 \end{equation}
 For the non-resonant term, thanks to \eqref{lem:commp}, we also know that 
 \begin{equation}
 \label{bd:pk}
 \mathbbm{1}_{\{t\notin I_{k,\eta}\cap I_{\ell,\xi}^c\}}\sqrt{\frac{k^2}{p_k(t,\eta)}}\lesssim \jap{|k-\ell,\eta-\xi|}^4\mathbbm{1}_{\{t\notin I_{k,\eta}\cap I_{\ell,\xi}^c\}}\sqrt{\frac{\ell^2}{p_\ell(t,\xi)}},
 \end{equation}
 where we paid an extra derivative on the low-frequency piece since $|k|/|\ell|\lesssim\jap{k-\ell}$. Moreover, having that 
 \begin{equation}
 |\ell,\xi|\lesssim \jap{t}|\ell,\xi-\ell t|,
 \end{equation}
 combining the bound above with \eqref{bd:mk}, \eqref{bd:pk}, Cauchy-Schwartz and Young's convolution inequality we arrive at 
 \begin{align}
 \label{bd:TNR1}
 \cT^{NR}_{\sfsym}(F_{\neq}^{Lo},G_{\neq}^{Hi},H)&\lesssim \jap{t}\norm{(-\Delta_L)^{-1}F_{\neq}}_{H^7}\norm{\sqrt{\frac{k^2}{p}}m|\nabla_L|G_{\neq}}_{L^2}\norm{H}_{L^2}\\
 &\lesssim \frac{1}{\jap{t}}\e^{-\delta_0\nu^\frac13 t}\norm{mF_{\neq}}_{L^2}\norm{\sqrt{\frac{k^2}{p}}m|\nabla_L|G_{\neq}}_{L^2}\norm{H}_{L^2}.
 \end{align}
In the last inequality we used Lemma \eqref{lem:elliptic} combined with the fact that $\jap{|k,\eta|}^9\lesssim \e^{-\delta_0 \nu^\frac13 t}m_k(t,\eta)$ since $N>10$. This bound is in agreement with \eqref{bd:Tsimneqneq}.
 
 We now turn our attention to the resonant part. From Lemma \ref{lem:commp} we deduce
  \begin{equation}
 \label{bd:pkR}
 \mathbbm{1}_{\{t\in I_{k,\eta}\cap I_{\ell,\xi}^c\}}\sqrt{\frac{k^2}{p_k(t,\eta)}}|\ell,\xi|\lesssim \jap{|k-\ell,\eta-\xi|}^4\mathbbm{1}_{\{t\in I_{k,\eta}\cap I_{\ell,\xi}^c\}}\frac{|\eta|}{k^2(1+|\frac{\eta}{k}-t|)}|\ell,\xi|\sqrt{\frac{\ell^2}{p_\ell(t,\xi)}}.
 \end{equation}
Since $t\in I_{k,\eta}$ we know that $t\approx |\eta|/|k|$, so we get
 \begin{equation}
\frac{|\eta||\ell,\xi|}{k^2}\lesssim \jap{\eta-\xi}\frac{|\eta|^2}{k^2}+\jap{k-\ell}\frac{|\eta|}{k}\frac{\jap{\xi}}{k^2}\lesssim \jap{|k-\ell,\eta-\xi|}^2\jap{t}^2.
\end{equation}
Recalling the definition of $m^d$ \eqref{def:md}, combining the bound above with \eqref{bd:pkR} we obtain 
  \begin{equation}
 \label{bd:pkR1}
 \mathbbm{1}_{\{t\in I_{k,\eta}\cap I_{\ell,\xi}^c\}}\sqrt{\frac{k^2}{p_k(t,\eta)}}|\ell,\xi|\lesssim \jap{t}^2\jap{|k-\ell,\eta-\xi|}^6\sqrt{\frac{\de_tm^d_k(t,\eta)}{m^d_k(t,\eta)}}\sqrt{\frac{\ell^2}{p_\ell(t,\xi)}}.
 \end{equation}
 Therefore, appealing again to Lemma \ref{lem:elliptic}, we have 
 \begin{align}
 \label{bd:TR1}
 \cT^{R}_{\sfsym}(F_{\neq}^{Lo},G_{\neq}^{Hi},H)&\lesssim \jap{t}^2\norm{(-\Delta_L)^{-1}F_{\neq}}_{H^8}\norm{\sqrt{\frac{k^2}{p}}mG_{\neq}}_{L^2}\norm{\sqrt{\frac{\de_t m^d}{m^d}}H}_{L^2}\\
 &\lesssim \e^{-\delta_0\nu^\frac13 t}\norm{mF_{\neq}}_{L^2}\norm{\sqrt{\frac{k^2}{p}}mG_{\neq}}_{L^2}\norm{\sqrt{\frac{\de_t m^d}{m^d}}H}_{L^2},
 \end{align}
 which is consistent with \eqref{bd:Tsimneqneq}.
 
 \medskip
 \noindent $\square$ \textit{Control of the high-low term}. By the definition of \eqref{def:cTsym} and a change of variables, observe that 
 \begin{equation}
 \cT_{\sfsym}(F_{\neq}^{Hi},G_{\neq}^{Lo},H)= \cT_{\sfsym}(\Delta_LG_{\neq}^{Lo}, \Delta_L^{-1}F_{\neq}^{Hi},H).
 \end{equation}
 Writing down this term explicitly, we obtain the bound 
 \begin{align}
 \cT_{\sfsym}(\Delta_LG_{\neq}^{Lo}, \Delta_L^{-1}F_{\neq}^{Hi},H)&\lesssim\sum_{k,\ell \in \mathbb{Z}}\iint_{\RR^2}\sqrt{\frac{k^2}{p_k(t,\eta)}}m_k(t,\eta)|k-\ell,\eta-\xi||\hat{G}^{Lo}_{\neq}|_{k-\ell}(\eta-\xi)\\
 \label{def:TRHLsym}&\qquad \times (\mathbbm{1}_{\{t\in I_{\ell,\xi}\}}+\mathbbm{1}_{\{t\notin I_{\ell,\xi}\}})\frac{|\ell,\xi|}{p_{\ell}(t,\xi)} |\hat{F}^{Hi}|_{\ell}(\xi)|\hat{H}|_{k}(\eta) \dd \eta\dd \xi\\
& :=\cJ^{R}+\cJ^{NR},
 \end{align}
 where $\cJ^{R}$ is the integral containing $t\in\mathbbm{1}_{\{t\in I_{\ell,\xi}\}}$ and $\cJ^{NR}$ the other one. Notice that with the change of variables we now have $\jap{|k,\eta|}\leq 3\jap{\ell,\xi}/2$. When $t\in I_{\ell,\xi}$, since $m^d,m^\nu,m^s$ are bounded Fourier multipliers,  we observe that 
 \begin{equation}
\mathbbm{1}_{\{t\in I_{\ell,\xi}\}}\frac{|\ell,\xi|}{p_{\ell}(t,\xi)}m_k(t,\eta)\lesssim \mathbbm{1}_{\{t\in I_{\ell,\xi}\}}\frac{|\xi|}{|\ell|^2}\frac{1}{1+|\frac{\xi}{\ell}-t|^2}\lesssim \jap{t}\sqrt{\frac{\de_tm^d_\ell(t,\xi)}{m^d_\ell(t,\xi)}}\sqrt{\frac{\ell^2}{p_\ell(t,\xi)}}m_\ell(t,\xi).
 \end{equation}
 Since $\sqrt{k^2/p}\approx\sqrt{\de_tm^d/m^d}$, moving this factor to $H$, we then deduce the bound 
 \begin{align}
 \mathcal{J}^R\lesssim \jap{t}\e^{-\delta_0\nu^\frac13t}\norm{mG_{\neq}}_{L^2}\norm{\sqrt{\frac{\de_tm^d}{m^d}}\sqrt{\frac{k^2}{p}}mF_{\neq}}_{L^2}\norm{\sqrt{\frac{\de_tm^d}{m^d}}H}_{L^2}.
 \end{align}
 When $t\notin I_{\ell,\xi}$ we have $|\xi/\ell-t|\gtrsim |\xi|/|\ell|^2$, hence 
  \begin{equation}
\mathbbm{1}_{\{t\notin I_{\ell,\xi}\}}\frac{|\ell,\xi|}{p_{\ell}(t,\xi)}m_k(t,\eta)\lesssim \mathbbm{1}_{\{t\notin I_{\ell,\xi}\}}\frac{|\xi|}{|\ell|^2}\frac{1}{1+|\frac{\xi}{\ell}-t|^2}m_\ell(t,\xi)\lesssim\sqrt{\frac{\ell^2}{p_\ell(t,\xi)}}m_\ell(t,\xi).
 \end{equation}
Thus
  \begin{align}
 \mathcal{J}^{NR}\lesssim \e^{-\delta_0\nu^\frac13t}\norm{mG_{\neq}}_{L^2}\norm{\sqrt{\frac{k^2}{p}}mF_{\neq}}_{L^2}\norm{\sqrt{\frac{\de_tm^d}{m^d}}H}_{L^2}.
 \end{align}
% Splitting as in \eqref{def:TRsym} and \eqref{def:TNRsym}, we have
% \begin{equation}
%\cT_{\sfsym}(\Delta_LG_{\neq}^{Lo}, \Delta_L^{-1}F_{\neq}^{Hi},H)\lesssim\cT_{\sfsym}^R(\Delta_LG_{\neq}^{Lo}, \Delta_L^{-1}F_{\neq}^{Hi},H)+\cT_{\sfsym}^{NR}(\Delta_LG_{\neq}^{Lo}, \Delta_L^{-1}F_{\neq}^{Hi},H).
% \end{equation}
% Arguing as done to obtain \eqref{bd:TNR1} we have 
% \begin{align}
% \cT_{\sfsym}^{NR}(\Delta_LG_{\neq}^{Lo}, \Delta_L^{-1}F_{\neq}^{Hi},H)&\lesssim   \jap{t}\norm{G_{\neq}}_{H^7}\norm{\sqrt{\frac{k^2}{p}}m|\nabla_L|(-\Delta_L)^{-1}F_{\neq}}_{L^2}\norm{H}_{L^2}\\
% &\lesssim \jap{t}\e^{-\delta_0\nu^\frac13t}\norm{mG_{\neq}}_{L^2}\norm{\sqrt{\frac{\de_t m^d}{m^d}}\sqrt{\frac{k^2}{p}}mF_{\neq}}_{L^2}\norm{H}_{L^2},
% \end{align}
% where in the last inequality we used 
% \begin{equation}
% \frac{|k,\eta-kt|}{p_k(t,\eta)}\leq \frac{1}{|k|(1+|\frac{\eta}{k}-t|)}\lesssim\sqrt{\frac{\de_t m^d_k(t,\eta)}{m^d_k(t,\eta)}}.
% \end{equation}
% Similarly, proceeding as done to obtain \eqref{bd:TR1}, we get 
%  \begin{align}
% \label{bd:TR2}
% \cT^{R}_{\sfsym}(\Delta_LG_{\neq}^{Lo}, \Delta_L^{-1}F_{\neq}^{Hi},H)&\lesssim \jap{t}^2\norm{G_{\neq}}_{H^8}\norm{\sqrt{\frac{k^2}{p}}m(-\Delta_L)^{-1}G_{\neq}}_{L^2}\norm{\sqrt{\frac{\de_t m^d}{m^d}}H}_{L^2}\\
% &\lesssim \jap{t}^2 \e^{-\delta_0\nu^\frac13 t}\norm{mG_{\neq}}_{L^2}\norm{\sqrt{\frac{\de_t m^d}{m^d}}\sqrt{\frac{k^2}{p}}mG_{\neq}}_{L^2}\norm{\sqrt{\frac{\de_t m^d}{m^d}}H}_{L^2}.
% \end{align}
The bound \eqref{bd:Tsimneqneq} is proved.

\medskip 

$\bullet$ \textit{Proof of \eqref{bd:Tsim0neq}:} first we notice that 
\begin{equation}
(\nabla^\perp\Delta_L^{-1}F)_0\cdot\nabla G_{\neq}=V_{F,0}^1\de_XG_{\neq}.
\end{equation}
Hence
\begin{align}
\cT_{\sfsym}(F_0,G_{\neq},H)\leq \sum_{k\in \ZZ}\iint_{\RR^2}\sqrt{\frac{k^2}{p_k(t,\eta)}}m_k(t,\eta)|\hat{V}^1_{F,0}|(\eta-\xi)|k||\hat{G}_{\neq}|_k(\xi)|\hat{H}|_k(\eta)\dd \eta \dd \xi.
\end{align}
Since $\jap{|k,\eta|}\lesssim \jap{|k,\xi|}+\jap{\eta-\xi}$ and $m^d,m^\nu,m^s$ are uniformly bounded, we deduce that 
\begin{equation}
m_k(t,\eta)\lesssim m_k(t,\xi)+\frac{\jap{\eta-\xi}^N}{\jap{|k,\xi|}^N}m_k(t,\xi).
\end{equation}
Hence
\begin{align}
&\cT_{\sfsym}(F_0,G_{\neq},H)\leq \,\sum_{k\in \ZZ}\iint_{\RR^2}\sqrt{\frac{k^2}{p_k(t,\eta)}}|\hat{V}^1_{F,0}|(\eta-\xi)|k||m(t)\hat{G}_{\neq}|_k(\xi)|\hat{H}|_k(\eta)\dd \eta \dd \xi \\
&\quad +\sum_{k\in \ZZ}\iint_{\RR^2}\sqrt{\frac{k^2}{p_k(t,\eta)}}\jap{\eta-\xi}^N|\hat{V}^1_{F,0}|(\eta-\xi)\frac{|k|}{\jap{|k,\xi|}^N}|m(t)\hat{G}_{\neq}|_k(\xi)|\hat{H}|_k(\eta)\dd \eta \dd \xi\\
&:=\cI_1+\cI_2.
\end{align}
Using \eqref{eq:pkp} and the definition of $m^d$,  we deduce 
\begin{equation}
\sqrt{\frac{k^2}{p_k(t,\eta)}}\leq \left(1+|\eta-\xi|\frac{1}{1+|\frac{\eta}{k}-t|}\right)\sqrt{\frac{k^2}{p_k(t,\xi)}}\lesssim\left(1+|\eta-\xi|\sqrt{\frac{\de_tm^d_k(t,\eta)}{m^d_k(t,\eta)}}\right)\sqrt{\frac{k^2}{p_k(t,\xi)}}.
\end{equation}
Hence
\begin{align}
\cI_2\lesssim \, &\norm{V_{F,0}^1}_{H^N}\norm{\sqrt{\frac{k^2}{p}}\jap{\cdot}^{-N}m\de_X G_{\neq}}_{H^3}\norm{H}_{L^2}\\
&+\norm{\de_YV_{F,0}^1}_{H^N}\norm{\sqrt{\frac{k^2}{p}}m G_{\neq}}_{L^2}\norm{\sqrt{\frac{\de_t m^d}{m^d}}H}_{L^2},
\end{align}
which is in agreement with \eqref{bd:Tsim0neq} since $N>10$. On $\cI_1$ we can pay regularity on $V_{F,0}^1$ and obtain the  bound 
\begin{align}
\cI_1\lesssim \, \norm{V_{F,0}^1}_{H^3}\norm{\sqrt{\frac{k^2}{p}}m\de_X G_{\neq}}_{L^2}\norm{H}_{L^2},
\end{align}
so \eqref{bd:Tsim0neq} is proved.
%\begin{align}
%\cT_{\sfsym}(F_0,G_{\neq},H)\lesssim \norm{V_{F,0}^1}_{H^N}\norm{\de_X\sqrt{\frac{k^2}{p}}mG_{\neq}}_{L^2}\norm{H}_{L^2}\\
%&+\norm{\de_YV_{F,0}^1}_{H^2}\norm{\sqrt{\frac{k^2}{p}}mG_{\neq}}_{L^2}\norm{\sqrt{\frac{\de_tm^d}{m^d}}H}_{L^2}
%\end{align}

\medskip 

$\bullet$ \textit{Proof of \eqref{bd:Tsimneq0}:} in this case we have 
\begin{equation}
(\nabla^\perp\Delta_L^{-1}F)_{\neq}\cdot\nabla G_{0}=(\de_X\Delta_L^{-1}F_{\neq})\de_YG_{0}.
\end{equation}
 Then we do a paraproduct decomposition to get
\begin{align}
\cT_{\sfsym}(F_{\neq},G_0,H)\leq\cT_{\sfsym}(F_{\neq}^{Hi},G_0^{Lo},H) + \cT_{\sfsym}(F_{\neq}^{Lo},G_0^{Hi},H).
\end{align}
For the high-low term, we move the factor $\sqrt{k^2/p}$ on $H$ and use  \eqref{bd:pneq0} to get
\begin{equation}
\cT_{\sfsym}(F_{\neq}^{Hi},G_0^{Lo},H)\lesssim \jap{\bigg( \sqrt{\frac{\de_tm^d}{m^d}}\sqrt{\frac{k^2}{p}}m|\hat{F}_{\neq}^{Hi}|*|\cF(\de_YG_0^{Lo})|\bigg),\sqrt{\frac{\de_tm^d}{m^d}}|H|}.
\end{equation}
Applying Cauchy-Schwartz and Young's convolution inequality we get a bound in agreement with \eqref{bd:Tsimneq0}. 

For the low-high term instead, we need to be careful in order to recover time-decay from $\de_X\Delta_L^{-1}$. This is because $x$-derivatives can be high in the $F^{Lo}$ piece since $G_0$ is concentrated on the zero $x$-frequencies. We then argue as follows: since $\xi$ is the high-frequency, we have
\begin{equation}
\jap{|k,\eta|}^N\lesssim  \jap{k}^N+\jap{\eta}^N\lesssim \jap{|k,\eta-\xi|}^N+\jap{\xi}^N.
\end{equation}
This implies
\begin{equation}
m_k(t,\eta)\lesssim m_k(t,\eta-\xi)+\frac{\jap{\xi}^N}{\jap{|k,\eta-\xi|}^N}m_k(t,\eta-\xi).
\end{equation}
Hence, using \eqref{bd:pneq0} we deduce that
\begin{align}
\cT_{\sfsym}(F_{\neq}^{Lo},G_0^{Hi},H)\lesssim \cJ_1+\cJ_2
\end{align}
where 
\begin{align}
& \cJ_1:=\sum_{k}\iint \left( \sqrt{\frac{\de_tm^d_k(t)}{m^d_k(t)}}\sqrt{\frac{k^2}{p_k(t)}}m_k(t)|\hat{F}^{Lo}_{\neq}|_{k}\right)(\eta-\xi)|\xi||\hat{G}_0^{Hi}|(\xi)\left(\sqrt{\frac{\de_tm^d_k(t)}{m^d_k(t)}}|\hat{H}|_k\right)(\eta)\dd \eta \dd \xi\\
&\cJ_2:=\sum_{k}\iint \frac{|k|m_k(t,\eta-\xi)}{\jap{|k,\eta-\xi|}^Np_k(t,\eta-\xi)}|\hat{F}^{Lo}_{\neq}|_{k}(\eta-\xi)\jap{\xi}^N|\xi||\hat{G}_0^{Hi}|(\xi)\left(\sqrt{\frac{\de_tm^d_k(t)}{m^d_k(t)}}|\hat{H}|_k\right)(\eta)\dd \eta \dd \xi
\end{align}
For $\cJ_1$ it is not difficult to get 
\begin{equation}
\cJ_1\lesssim \norm{G_0}_{H^{3}}\norm{\sqrt{\frac{\de_tm^d}{m^d}}\sqrt{\frac{k^2}{p}}mF_{\neq}}_{L^2}\norm{\sqrt{\frac{\de_t m^d}{m^d}}H}_{L^2}.
\end{equation}
For $\cJ_2$ instead, we have 
\begin{align}
\cJ_2&\lesssim \norm{\de_X(-\Delta_L)^{-1}m F_{\neq}}_{H^{-N+2}}\norm{\de_YG_0}_{H^N}\norm{\sqrt{\frac{\de_t m^d}{m^d}}H}_{L^2}\\
&\lesssim \frac{1}{\jap{t}^2}\norm{mF_{\neq}}_{L^2}\norm{\de_YG_0}_{H^N}\norm{\sqrt{\frac{\de_t m^d}{m^d}}H}_{L^2},
\end{align}
where in the last line we used \eqref{lem:elliptic} and $N>10$. The bound \eqref{bd:Tsimneq0} is then proved.

\medskip

$\bullet$ \textit{Proof of \eqref{bd:Ssimneqneq}:} notice that, since we have $\de_tp/p$ in front of $\hat{F}$, we always have $F_{\neq}$. It is also enough to prove the bound for 
\begin{equation}
\cS_{\sfsym}^1(F,G_\neq,H)= \left|\jap{\sqrt{\frac{k^2}{p}}m\left(\frac{\de_t p}{p}\hat{F}_{\neq}*G_{\neq}\right),H}\right|.
\end{equation}
Using that $|\de_t p/p|\leq 2 \sqrt{k^2/p}$ and the algebra property of $H^N$, we get 
\begin{equation}
\cS_{\sfsym}^1(F,G_\neq,H)\lesssim \e^{-\delta_0\nu^\frac13t}\norm{\sqrt{\frac{k^2}{p}}mF_{\neq}}_{L^2}\norm{mG_{\neq}}_{L^2}\norm{\sqrt{\frac{\de_tm^d}{m^d}}H}_{L^2},
\end{equation}
whence proving \eqref{bd:Ssimneqneq}.

\medskip 

$\bullet$ \textit{Proof of \eqref{bd:Ssimneq0}:} in this case $\cS=\cS^1$ defined above. Then, analogously to what we have done to treat $\cT_{\sfsym}(F_\neq,G_0,H)$, we use the paraproduct decomposition first 
\begin{equation}
\cS_{\sfsym}^1(F,G_0,H)\leq \cS_{\sfsym}^1(F^{Hi},G_0^{Lo},H)+\cS_{\sfsym}^1(F^{Lo},G_0^{Hi},H).
\end{equation}
For the high-low piece, we can proceed as done in the proof of \eqref{bd:Ssimneqneq} to get 
\begin{equation}
\label{bd:SsymHL}
\cS_{\sfsym}^1(F^{Hi},G_0^{Lo},H)\lesssim\norm{\sqrt{\frac{k^2}{p}}mF_{\neq}}_{L^2}\norm{G_{0}}_{H^3}\norm{\sqrt{\frac{\de_tm^d}{m^d}}H}_{L^2}
\end{equation}
For the low-high piece, we argue as done for the low-high term in the proof of \eqref{bd:Tsimneq0}. Namely, we can split the derivatives with higher-order in $x$ and $y$. In the first case, namely the term corresponding to $\cJ_1$ in the proof of \eqref{bd:Tsimneq0}, we argue as done for the low-high term and we prove the same bound as in \eqref{bd:SsymHL}. In the other case, we proceed as done for  $\cJ_2$ in the proof of \eqref{bd:Tsimneq0}. Overall, we get 
\begin{align}
\label{bd:SsymLH}
\cS_{\sfsym}^1(F^{Lo},G_0^{Hi},H)\lesssim\,&\norm{\sqrt{\frac{k^2}{p}}mF_{\neq}}_{L^2}\norm{G_{0}}_{H^3}\norm{\sqrt{\frac{\de_tm^d}{m^d}}H}_{L^2}\\
&+ \norm{\frac{\de_tp}{p}m F_{\neq}}_{H^{-N+2}}\norm{G_0}_{H^N}\norm{\sqrt{\frac{\de_t m^d}{m^d}}H}_{L^2}.
\end{align}
Having that $\de_tp\leq \jap{t}\jap{|k,\eta|}^2$, using again Lemma \ref{lem:elliptic}  we have
\begin{equation}
 \norm{\frac{\de_tp}{p}m F_{\neq}}_{H^{-N+2}}\lesssim \jap{t}\norm{(-\Delta_L)^{-1}mF_{\neq}}_{H^{-N+4}}\lesssim\frac{1}{\jap{t}}\norm{mF_{\neq}}_{H^{-N+6}},
\end{equation}
which proves \eqref{bd:Ssimneq0} since $N>6$.

\end{proof}

\subsection{Bounds for the higher order energy}
The structure of the proof for the higher--order energy is analogous to what we have done for $\sfE_{\sfsym}$. However, bounds will be simpler because we do not have to exchange frequencies for the unbounded multiplier $\sqrt{k^2/p}$. We define the transport and stretching nonlinear terms as
\begin{align}
\label{def:cTho}
&\cT_{\sfho}(F,G,H):=\left|\jap{m\cF\left(\nabla^\perp\Delta^{-1}_L F\cdot \nabla G \right), \hat{H}}\right|,\\
\label{def:cSho}
&\cS_{\sfho}(F,G,H):=\left|\jap{m\bigg(\frac{\de_t p}{p}\hat{F}*\big(\hat{G}-2\frac{\ell^2}{p}\hat{G}\big)\bigg),H}\right|.
\end{align}
We have the following.
\begin{lemma}
\label{lem:keyEho}
Let $m$ be the Fourier multiplier defined in \eqref{def:m} with $N>10$. Then: 
\begin{align}
\label{bd:Thoneqneq}
\cT_{\sfho}(F_{\neq},G_{\neq},H)\lesssim \e^{-\delta_0\nu^\frac13t} \norm{\sqrt{\frac{k^2}{p}}mF_{\neq}}_{L^2}\norm{\nabla_LmG_{\neq}}_{L^2}\norm{H}_{L^2}.
\end{align}
Denoting $(\nabla^\perp\Delta_L^{-1}F)_0=(V_{F,0}^1,0)$, one has
\begin{align}
\label{bd:Tho0neq}
\cT_{\sfho}(F_{0},G_{\neq},H)\lesssim\,& \norm{V_{F,0}^1}_{H^N}\norm{m\de_X G_{\neq}}_{L^2}\norm{H_{\neq}}_{L^2}\\
\label{bd:Thoneq0}
\cT_{\sfho}(F_{\neq},G_{0},H)\lesssim\, &\norm{H}_{L^2}\bigg(\norm{G_0}_{H^{3}}\norm{\sqrt{\frac{\de_tm^d}{m^d}}\sqrt{\frac{k^2}{p}}mF_{\neq}}_{L^2}\\
&\qquad+\frac{1}{\jap{t}^2}\norm{mF_{\neq}}_{L^2}\norm{\de_YG_0}_{H^{N}}\bigg)
\end{align}
For the stretching nonlinearities we have:
\begin{align}
\label{bd:Shoneq}
&\cS_{\sfsym}(F,G_{\neq},H)\lesssim\, \e^{-\delta_0\nu^\frac13t}\norm{\sqrt{\frac{k^2}{p}}mF_{\neq}}_{L^2}\norm{mG_{\neq}}_{L^2}\norm{H}_{L^2},\\
\label{bd:Sho0}
&\cS_{\sfsym}(F,G_{0},H)\lesssim\,\left(\norm{\sqrt{\frac{k^2}{p}}mF_{\neq}}_{L^2}\norm{G_{0}}_{H^3} + \frac{1}{\jap{t}}\norm{mF_{\neq}}_{L^2}\norm{G_{0}}_{H^N}\right)\norm{H}_{L^2}.
\end{align}
\end{lemma}
\begin{proof}[Proof of Lemma \ref{lem:keyEho}]
To prove \eqref{bd:Thoneqneq}, we simply observe  that
\begin{equation}
\nabla^\perp \Delta_L^{-1} F_{\neq}\cdot \nabla G_{\neq}=\nabla^\perp_L \Delta_L^{-1} F_{\neq}\cdot \nabla_L G_{\neq}, \qquad \text{and} \qquad \norm{\nabla_L\Delta_L^{-1}F_{\neq}}_{L^2}\lesssim \norm{\sqrt{\frac{k^2}{p}}F_{\neq}}_{L^2}.
\end{equation}
Being $m^d,m^\nu,m^s$ bounded Fourier multipliers, using the algebra property of $H^N$ we deduce that 
\begin{align}
\cT_{\sfho}(F_{\neq},G_{\neq},H)&\lesssim\e^{-\delta_0\nu^\frac13t}\norm{\nabla_L\Delta_L^{-1}mF_{\neq}}_{L^2}\norm{\nabla_LmG_{\neq}}_{L^2}\norm{H}_{L^2}
\end{align}
whence proving \eqref{bd:Thoneqneq}.

Turning our attention to \eqref{bd:Tho0neq}, we first observe that
\begin{align}
\cT_{\sfho}(F_{0},G_{\neq},H)=\left|\jap{m\cF\left(V_{F,0}^1\de_X G_{\neq} \right), \hat{H}_{\neq}}\right|
\end{align}
since $\jap{m\cF\left(V_{F,0}^1\de_X G_{\neq} \right), \hat{H}_{0}}=-\jap{V_{F,0}^1G , m\de_X H_{0}}=0$. The proof \eqref{bd:Tho0neq} easily follows as an application of Cauchy-Schwartz and Young's inequality.

The proof of the bounds \eqref{bd:Thoneq0}--\eqref{bd:Sho0} is identical to the ones for \eqref{bd:Tsimneq0}--\eqref{bd:Ssimneq0}. This is because in the latter bounds we have only moved the factor $\sqrt{k^2/p}$ on the function $H$.
\end{proof}
With Lemma \eqref{lem:keyEho} at hand, we show how to improve  \eqref{Bho}.
\begin{proof}[Proof: improvement of \eqref{Bho}]
For the transport nonlinearity \eqref{def:Tho}, recall that
\begin{equation}
\norm{m(\Omega_{\neq},J_{\neq})}_{L^2}\lesssim \nu^{-\frac16}\sqrt{\sfD_{\sfho}}.
\end{equation}
Hence, combining \eqref{bd:Thoneqneq}--\eqref{bd:Thoneq0}, since $\sqrt{k^2/p}(\hat{\Omega},\hat{J})=(Z,Q)$, we have
\begin{align}
\sfT_{\sfho}\lesssim\,& \nu^{-\frac12}\e^{-\delta_0 \nu^\frac13t}\sqrt{\sfE_{\sfsym}}\sqrt{\sfE_{\sfho}}\sqrt{\sfD_{\sfho}}+\nu^{-\frac12-\frac16}\sqrt{\sfE_0}\sfD_{\sfho}\\
&\, +\left(\sqrt{\sfE_{0}}\sqrt{\sfD_{\sfsym}}+\frac{1}{\jap{t}}\nu^{-\frac16}\sqrt{\sfD_{\sfho}}\sqrt{\sfD_0}\right)\sqrt{\sfE_{\sfho}}.
\end{align}
Similarly, from \eqref{bd:Shoneq}--\eqref{bd:Sho0}, using that 
\begin{align}
\frac{1}{\jap{t}}&\norm{m(\Omega_\neq,J_\neq)}_{L^2}\norm{(\Omega_0,J_0)}_{H^N}\norm{m(\Omega,J)}_{L^2}\\
&=\frac{1}{\jap{t}}\norm{m(\Omega_\neq,J_\neq)}_{L^2}\norm{\de_Y(U_0^1,B_0^1)}_{H^N}\norm{m(\Omega,J)}_{L^2}\lesssim \frac{1}{\jap{t}}\nu^{-\frac16-\frac12}\sqrt{\sfD_{\sfho}}\sqrt{\sfD_{0}}\sqrt{\sfE_{\sfho}}
\end{align}
we have 
\begin{align}
\sfS_{\sfho}\lesssim\,& \nu^{-\frac16}\e^{-\delta_0 \nu^\frac13t}\sqrt{\sfE_{\sfsym}}\sqrt{\sfE_{\sfho}}\sqrt{\sfD_{\sfho}}+\sqrt{\sfE_{\sfho}}\sqrt{\sfE_{\sfsym}}\sqrt{\sfE_{0}}+\frac{1}{\jap{t}}\nu^{-\frac23}\sqrt{\sfD_{\sfho}}\sqrt{\sfD_{0}}\sqrt{\sfE_{\sfho}}.
\end{align}
Integrating in time and applying the bootstrap hypothesis, we get 
\begin{align}
\int_0^t\sfT_{\sfho}+\sfS_{\sfho}\dd \tau\lesssim\, &\eps^2 \nu^{-\frac12}\int_0^t \jap{\tau}\e^{-\delta_0\nu^\frac13\tau}\sqrt{\sfD_{\sfho}}\dd \tau+\eps\nu^{-\frac23}\int_0^t\sfD_{\sfho}\dd \tau\\
&+\eps^2\int_0^t \jap{\tau}\sqrt{\sfD_{\sfsym}}\dd \tau+\eps \nu^{-\frac16}\int_0^t \sqrt{\sfD_{\sfho}}\sqrt{\sfD_0}\dd \tau\\
&+\eps^3\int_0^t\jap{\tau}\dd \tau+\eps\nu^{-\frac23}\int_0^t\sqrt{\sfD_{\sfho}}\sqrt{\sfD_{0}}\dd \tau.
\end{align}
Applying Cauchy-Schwartz inequality and the bootstrap hypotheses, we get 
\begin{align}
&\eps^2 \nu^{-\frac12}\int_0^t \jap{\tau}\e^{-\delta_0\nu^\frac13\tau}\sqrt{\sfD_{\sfho}}\dd \tau\lesssim \eps^{3}\jap{t}\nu^{-\frac12}\left(\int_0^t\jap{\tau}^2\e^{-2\delta_0\nu^\frac13\tau}\dd\tau\right)^\frac12\lesssim(\eps\nu^{-\frac23})\eps^{2}\jap{t}^2,\\
&\eps^2\int_0^t \jap{\tau}\sqrt{\sfD_{\sfsym}}\dd \tau\lesssim \eps^3 \jap{t}^\frac32,\\
&\eps \nu^{-\frac16}\int_0^t \sqrt{\sfD_{\sfho}}\sqrt{\sfD_0}\dd \tau\lesssim (\eps \nu^{-\frac16})\eps^2\jap{t},\\
&\eps\nu^{-\frac23}\int_0^t\sqrt{\sfD_{\sfho}}\sqrt{\sfD_{0}}\dd \tau\lesssim (\eps \nu^{-\frac23})\eps^2\jap{t}.
\end{align}
Integrating in time \eqref{bd:dtEho}, using again the bootstrap hypothes, when $\eps\ll \nu^\frac23$ we get 
\begin{align}
\sfE_{\sfho}+\frac{1}{16}\int_0^t \sfD_{\sfho}\dd \tau&\leq 4\int_0^t\sqrt{\sfE_{\sfsym}}\sqrt{\sfE_{\sfho}}\dd \tau+\eps^2\jap{t}^2\leq 4\sqrt{10}\sqrt{C_1}\eps^2\int_0^t\jap{t}\dd \tau+\eps^2\jap{t}^2\\
&\leq (8\sqrt{10}\sqrt{C_1}+1)\eps^2\jap{t}^2.
\end{align}
It is then enough that 
\begin{equation}
8\sqrt{10}\sqrt{C_1}+1\leq \frac{C_1}{2},
\end{equation}
which is certainly true for $C_1=4000$.
\end{proof}

\subsection{Bounds on the zero modes.} We finally show how to improve \eqref{B0}.
\begin{proof} Using the uniform boundedness of $m^d,m^\nu,m^s$ and \eqref{bd:pneq0}, we first observe that 
\begin{align}
\label{bd:U2}&\norm{(U^2_{\neq},B^2_{\neq})}_{H^N}=\norm{\de_X\Delta_L^{-1}(\Omega_{\neq},J_{\neq})}_{H^N}\lesssim \e^{-\delta_0\nu^\frac13t} \norm{\sqrt{\frac{\de_tm^d}{m^d}}m(Z,Q)}_{L^2},\\
\label{bd:U1}&\norm{\de_X(U^1_{\neq},B^1_{\neq})}_{H^N}=\norm{\de_X(\de_Y-t\de_X)\Delta_L^{-1}(\Omega_{\neq},J_{\neq})}_{H^N}\lesssim \e^{-\delta_0\nu^\frac13t} \norm{m(Z,Q)}_{L^2}.
\end{align}
where we also used $|\de_t p/p|\lesssim \sqrt{k^2/p}$. Applying Cauchy-Schwartz and the algebra property of $H^N$ we see that we can bound $\sfR_{\neq}$ in \eqref{def:Rneq} by 
\begin{align}
\sfR_\neq\lesssim \,&\norm{(U^2_{\neq},B^2_{\neq})}_{H^N}\norm{(U^1_{\neq},B^1_{\neq})}_{H^N}\norm{\de_Y(U_0^1,B_0^1)}_{H^N}\\
&+\frac{1}{\jap{t}^2}\norm{(U^2_{\neq},B^2_{\neq})}_{H^N}\norm{(\Omega_{\neq},J_{\neq})}_{H^N}\norm{\de_Y(\Omega_0,J_0)}_{H^N}\\
&+\frac{1}{\jap{t}^2}\norm{\de_X(U^1_{\neq},B^1_{\neq})}_{H^N}\norm{(\Omega_{\neq},J_{\neq})}_{H^N}\norm{J_0}_{H^N}.
\end{align}
Combining the bound above with \eqref{bd:U2}--\eqref{bd:U1} and using the boostrap hypothesis, we get
\begin{align}
\sfR_{\neq}&\lesssim \e^{-2\delta_0\nu^\frac13t} \bigg(\nu^{-\frac12}\sqrt{\sfE_{\sfsym}}\sqrt{\sfD_{\sfsym}}\sqrt{\sfD_{0}}+\frac{\nu^{-\frac12}}{\jap{t}}\sqrt{\sfE_{\sfho}}\sqrt{\sfD_{\sfsym}}\sqrt{\sfD_{0}}+\frac{1}{\jap{t}}\sqrt{\sfE_{\sfsym}}\sqrt{\sfE_{\sfho}}\sqrt{\sfE_{0}}\bigg)\\
&\lesssim (\eps\nu^{-\frac12})(\sfD_{\sfsym}+\sfD_{0})+\eps^3\e^{-2\delta_0\nu^\frac13t}.
\end{align} 
Integrating \eqref{bd:dtE0} in time, we get 
\begin{equation}
\sfE_0+\int_0^t\sfD_0\dd \tau\leq \eps^2(1+C(\eps \nu^{-\frac12}+\eps\nu^{-\frac13})),
\end{equation}
for some constant $C>0$. Therefore, when $\eps \ll \nu^\frac23$ we see that we can improve from the constant $100$ in \eqref{B0} to $50$ as desired, whence completing the proof of Proposition \ref{prop:boot}. 
\end{proof}

\subsection*{Acknowledgments} The author would like to thank Ruizhao Zi for sharing their result \cite{chen23}. The research of MD was supported by the SNSF Grant 182565, by the Swiss State
Secretariat for Education, Research and lnnovation (SERI) under contract number M822.00034 and by the GNAMPA-INdAM . 
\bibliographystyle{siam}
\bibliography{MHDbiblio}

\end{document}